\documentclass[a4paper]{article}
\usepackage[a4paper, total={6in, 9in}]{geometry}
\usepackage{graphicx,multicol,float,amsmath,amsthm,amsfonts,dsfont,amssymb,amsthm,MnSymbol,stmaryrd,color,url}
\usepackage[all,2cell, cmtip]{xy} \UseAllTwocells \SilentMatrices
\usepackage[hidelinks]{hyperref}
\usepackage[titletoc]{appendix}
\usepackage[bbgreekl]{mathbbol}
\usepackage{verbatim}
\usepackage{mathtools}
\usepackage{eucal}

\newtheorem{thm}{Theorem}[section]
\newtheorem{thmintro}{Theorem}

\newtheorem{lem}[thm]{Lemma}
\newtheorem{prop}[thm]{Proposition}
\newtheorem{cor}[thm]{Corollary}

\theoremstyle{definition}
\newtheorem{defn}[thm]{Definition}
\theoremstyle{remark}

\newtheorem{remark}[thm]{Remark}
\newtheorem{exam}[thm]{Example}

\newtheoremstyle{outlinenotes}{}{}{\color{blue}}{}{\color{blue}\bfseries}{}{ }{}
\theoremstyle{outlinenotes}

\newtheoremstyle{todonotes}{}{}{\color{red}}{}{\color{red}\bfseries}{}{ }{}
\theoremstyle{todonotes}

\newcommand{\Deltaac}{\Delta_{\rm ac}}
\newcommand{\Deltain}{\Delta_{\rm in}}
\newcommand{\Simp}{\mathbb{\Delta}}
\newcommand{\Simpop}{\Simp^{\rm op}}
\newcommand{\Simpac}{\Simp_{\rm ac}}
\newcommand{\Simpin}{\Simp_{\rm in}}
\newcommand{\Set}{{\rm Set}}

\newcommand{\sSet}{\Set_\Delta}

\newcommand{\Fin}{{\rm Fin}}
\newcommand{\cFin}{\mathcal{F}{\rm in}}
\newcommand{\sSetf}{\Fin_\Delta}
\newcommand{\sSetfop}{(\sSetf)^{\rm op}}
\newcommand{\csSetf}{\cFin_\Simp}
\newcommand{\csSetfop}{(\csSetf)^{\rm op}}

\newcommand{\Finpt}{\Fin_*}

\newcommand{\cFinpt}{\cFin_*}
\newcommand{\FinDel}{\Finpt\times\Delta^{\rm op}}
\newcommand{\cFinDel}{\cFinpt \times \Simpop}
\newcommand{\Cat}{{\rm Cat}}

\newcommand{\Kan}{{\rm Kan}}
\newcommand{\qCat}{{\rm qCat}}
\newcommand{\Catoo}{\mathcal{C}{\rm at}_\infty}
\newcommand{\CSS}{{\rm CSS}}
\newcommand{\cCSS}{\mathcal{C}{\rm SS}}
\newcommand{\Bicatoo}{\mathcal{C}{\rm at}_{\inftwo}}

\newcommand{\Alg}{{\rm Alg}}
\newcommand{\cAlg}{\mathcal{A}{\rm lg}}
\newcommand{\csmAlg}{\cAlg^\amalg}

\newcommand{\cCoalg}{\mathcal{C}{\rm oalg}}
\newcommand{\csmCoalg}{\cCoalg^\amalg}

\newcommand{\Fun}{{\rm Fun}}
\newcommand{\cFun}{\mathcal{F}{\rm un}}
\newcommand{\Map}{{\rm Map}}
\newcommand{\cMap}{\mathcal{M}{\rm ap}}

\def\cSpan#1{\mathcal{S}{\rm pan}\left({#1}\right)}

\def\csmtSpan#1{\mathcal{S}{\rm pan}_2^\times \left({#1}\right)}
\def\cSeg#1{\mathcal{S}{\rm eg}\left({#1}\right)}
\def\cSego#1{\mathcal{S}{\rm eg}_0\left({#1}\right)}
\def\Pyr#1{\Sigma^{#1}}
\def\cPyr#1{\mathbb{\Sigma}^{#1}}
\def\Wedge#1{\Lambda^{#1}}
\def\cWedge#1{\mathbb{\Lambda}^{#1}}

\def\Cart#1{\Pi({#1})}
\def\cCart#1{\mathbb{\Pi}(#1)}
\def\Cartne#1{\Delta_{\mathrm{inj}/{#1}}}
\def\Cartop#1{\Pi({#1})^{\rm op}}
\def\cCartop#1{\mathbb{\Pi}(#1)^{\rm op}}
\def\Sing#1{P({#1})}
\def\cSing#1{\mathcal{P}({#1})}
\def\Singop#1{P({#1})^{\rm op}}
\def\cSingop#1{\mathcal{P}({#1})^{\rm op}}

\def\Unstr#1{\left\langle{#1}\right\rangle}
\def\sp#1{{\rm sp}\left(#1\right)}

\newcommand{\cA}{\mathcal{A}}
\newcommand{\cB}{\mathcal{B}}

\newcommand{\cC}{\mathcal{C}}
\newcommand{\fC}{\mathfrak{C}}
\newcommand{\cD}{\mathcal{D}}

\newcommand{\fG}{\mathfrak{G}}

\newcommand{\cM}{\mathcal{M}}

\newcommand{\cS}{\mathcal{S}}

\newcommand{\cU}{\mathcal{U}}
\newcommand{\cX}{\mathcal{X}}

\newcommand{\inftwo}{(\infty,2)}
\newcommand\Sp{\mathcal{S}{\rm p}}
\newcommand\Un{{\rm Un}}
\newcommand{\actmorR}{\rightarrow\Mapsfromchar}
\newcommand{\actmorL}{\Mapstochar\leftarrow}
\newcommand\inmorR{\rightarrowtail}

\newcommand\op{{\rm op}}

\def\Simplex#1{\Delta\left[{#1}\right]}

\def\Simplexn#1{\nabla\left[{#1}\right]}

\newcommand{\id}{{\rm Id}}
\renewcommand{\lim}{\operatornamewithlimits{lim}}
\newcommand{\colim}{\operatornamewithlimits{colim}}

\newcommand{\adjRelayII}[3][2.2em]{\ensuremath{\SelectTips{cm}{10}\xymatrix@C=#1@1{{#2} \ar@<1ex>[r]^-{\ArgI}^-{}="1" & {#3} \ar@<1ex>[l]^-{\ArgII}^-{}="2" \ar@{}"1";"2"|(.3){\hbox{}}="7" \ar@{}"1";"2"|(.7){\hbox{}}="8" \ar@{|-} "8" ;"7"}}}
\newcommand{\radj}[1][]{\def\ArgI{#1}\radjRelayI}
\newcommand{\radjRelayI}[1][]{\def\ArgII{#1}\radjRelayII}
\newcommand{\radjRelayII}[3][2.2em]{\ensuremath{\SelectTips{cm}{10}\xymatrix@C=#1@1{{#2} \ar@<-1ex>[r]_-{\ArgI}^-{}="1" & {#3} \ar@<-1ex>[l]_-{\ArgII}^-{}="2" \ar@{}"1";"2"|(.3){\hbox{}}="7" \ar@{}"1";"2"|(.7){\hbox{}}="8" \ar@{|-} "7" ;"8"}}}

\numberwithin{equation}{section}

\title{Simplicial spaces, lax algebras and the $2$-Segal condition}
\author{Mark D Penney\footnote{ email: \href{mailto:mpenney@mpim-bonn.mpg.de}{mpenney@mpim-bonn.mpg.de}}}
\date{}

\begin{document}
 \maketitle
 \begin{abstract}
 Dyckerhoff--Kapranov \cite{DK12} and G\'alvez-Carrillo--Kock--Tonks \cite{KockI} independently introduced the notion of a {\em $2$-Segal space}, that is, a simplicial space satisfying $2$-dimensional analogues of the Segal conditions, as a unifying framework for understanding the numerous Hall algebra-like constructions appearing in algebraic geometry, representation theory and combinatorics. In particular, they showed that every $2$-Segal object defines an algebra object in the {\em $\infty$-category of spans}. 
 
 In this paper we show that this algebra structure is inherited from the {\em initial simplicial object} $\Simplex{\bullet}$. Namely, we show that the standard $1$-simplex $\Simplex{1}$ carries a {\em lax algebra} structure. As a formal consequence the space of $1$-simplices of a simplicial space is also a lax algebra. We further show that the $2$-Segal conditions are equivalent to the associativity of this lax algebra.
 \end{abstract}
 
 \tableofcontents

 \section{Introduction}
 \label{Intro}

  Simplicial objects satisfying the Segal conditions \cite{RezkCSS} are used throughout homotopy theory and higher category theory to encode coherently associative algebras. Recently, Dyckerhoff--Kapranov \cite{DK12} and G\'alvez-Carrillo--Kock--Tonks \cite{KockI} independently introduced a generalisation of Rezk's Segal conditions, the {\em $2$-Segal conditions}. They showed that simplicial objects satisfying the $2$-Segal conditions encode algebra objects in {\em $\infty$-categories of spans}. In this paper we shall elucidate the exact role played by the $2$-Segal condition in the construction of these algebra objects.
  
  Specifically, we provide a novel construction of the algebra in the $\infty$-category of spans associated to a $2$-Segal object. We do so by exploiting the fact that the initial simplicial object in an $\infty$-category having finite limits is 
  \begin{equation*}
  \Simplex{\bullet}: \Simpop \to \csSetfop,
  \end{equation*}
where $\csSetf$ is the nerve of the category of level-wise finite simplicial sets. The following diagram
\begin{equation}
\label{DeltaProd}
 \includegraphics[height=1.45cm]{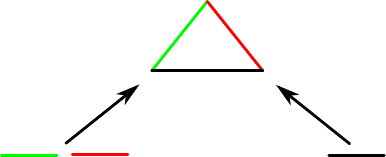}
\end{equation}
endows the standard $1$-simplex $\Simplex 1$ with a product $\mu$ as an object of the $\infty$-category of spans in $\csSetfop$. While the product $\mu$ fails to be associative, the diagram
\begin{equation*}
 \includegraphics[height=4.5cm]{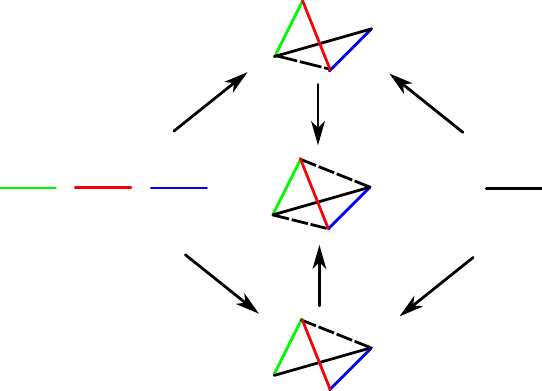}
\end{equation*}
defines a non-invertible $2$-morphism in the {\em $\inftwo$-category of bispans} \cite{RuneSpans} in $\csSetfop$ from $\mu \circ (\mu\times\id)$ to $\mu \circ ( \id\times\mu)$ which witnesses its {\em lax associativity}. We show the following:

\begin{thmintro}
\label{SimplexSumm}
 The standard $1$-simplex $\Simplex 1$ is a {\em lax algebra} in the $\inftwo$-category of bispans in $\csSetfop$ with product $\mu$.
\end{thmintro}

From any simplicial object $\cX_\bullet\in \cC_\Simp$ in an $\infty$-category $\cC$ having finite limits one has a finite limit preserving functor $\cX:\csSetfop \to \cC$ given by right Kan extension,
\begin{equation*}
 \xymatrixrowsep{1.1pc} \xymatrix{ \csSetfop \ar[r]^-\cX & \cC \\
 \Simpop \ar@{^{(}->}[u] \ar[ru]_-\cX & }
\end{equation*}
The simplicial object $\cX_\bullet$ is then the image of $\Simplex{\bullet}$ under the right Kan extension. As a formal consequence of Theorem \ref{SimplexSumm} one obtains a lax algebra structure on $\cX_1$ as an object of the $\inftwo$-category of bispans in $\cC$. We show that the $2$-Segal condition is equivalent to the associativity of this lax algebra. 

\begin{thmintro}
 \label{MainSummAlg} 
 The object of $1$-simplices $\cX_1$ of a simplicial object $\cX \in \cC_\Simp$ is canonically a lax algebra in the $\inftwo$-category of bispans in $\cC$. Furthermore, $\cX_\bullet$ satisfies the $2$-Segal conditions if and only if $\cX_1$ is an algebra object.  
\end{thmintro}

To our knowledge there are two previous constructions of the above algebra associated to a $2$-Segal object by different techniques and at differing levels of generality. The first is due to Dyckerhoff--Kapranov, who associate to each injectively fibrant $2$-Segal object $\cX$ in a combinatorial simplicial model category a monad on the $\inftwo$-category of bispans which makes $\cX_1$ an algebra when $\cX_0\simeq \ast$ (\cite{DK12} 11.1). G\'alvez-Carrillo--Kock--Tonks showed the space of $1$-simplices in a $2$-Segal space carries an algebra structure using a novel equivalence between simplicial spaces and certain monoidal functors (\cite{KockI} 7.4).

By first constructing the lax algebra for the initial simplicial object $\Simplex{\bullet}$ we are able to work at the highest level of generality possible, that of $2$-Segal objects in an $\infty$-category having finite limits. Furthermore, this paper is the first in a series of papers which build new examples of higher categorical bialgebras using $2$-Segal spaces \cite{penney2017bialg, penney2017bimon}. The construction given here provides the groundwork for these later papers.

\paragraph{Outline.} The paper begins in Section \ref{Sec:LaxAlg} with the necessary background for defining the notion of a {\em lax algebra object} in a symmetric monoidal $\inftwo$-category. As the informal description of a lax algebra given above explicitly makes use of non-invertible $2$-morphisms it should come as no surprise that the definition makes use of ideas which are inherently $\inftwo$-categorical. Specifically, one needs {\em lax functors} between $\inftwo$-categories. In the particular model of $\inftwo$-categories that we use in this paper lax functors are defined in terms of the {\em unstraightening construction}, the $\infty$-categorical generalisation of the ordinary Grothendieck construction. A lax algebra object in a symmetric monoidal $\inftwo$-category is then defined to be a symmetric monoidal lax functor out of the category $\csmAlg$ which corepresents algebra objects. In Section \ref{Sec:inftwo} we review our chosen model of $\inftwo$-categories, in Section \ref{Sec:laxfun} we define lax functors in terms of the unstraightening construction and in Section \ref{Sec:algprop} we introduce the category $\csmAlg$.

As indicated in Theorems \ref{SimplexSumm} and \ref{MainSummAlg} we will be constructing lax algebra objects in symmetric monoidal $\inftwo$-categories of bispans, which are the subjects of Section \ref{Sec:bispans}. One of the technical hurdles that one must surmount when working with lax functors of $\inftwo$-categories is the difficulty of providing an explicit description of the unstraightening construction. As such, we devote a considerable part of this section to the unstraightening of the $\inftwo$-category of bispans, rendering it into a form which is adequate for our purposes. We begin, in Section \ref{Sec:twisted}, with a brief discussion on the twisted arrow construction, a construction which appears throughout this work. Next we review Haugseng's original definition of the $\inftwo$-category of bispans in Section \ref{BispanConst}. Finally, Sections \ref{StrictBispan} and \ref{UnstrSpan} contain technical material culminating in a description of the unstraightening of the $\inftwo$-category of bispans.

While the previous two sections have been essentially preparatory, Section \ref{Sec:multialg} presents the main results of the paper. We begin in Section \ref{Sec:algcomb} by proving Theorem \ref{SimplexSumm}, that is, by giving an explicit construction of a lax algebra structure on $\Simplex{1}$. The first statement in Theorem \ref{MainSummAlg} is proven in Section \ref{Sec:intospaces} as a formal consequence of the construction in the previous section. Finally, in Section \ref{Sec:SegAssoc} we prove the second statement of Theorem \ref{MainSummAlg} connecting the $2$-Segal condition to the associativity of the lax algebras constructed in Section \ref{Sec:intospaces}.

\paragraph{Acknowledgements.} We would like to thank Tobias Dyckerhoff for a number of insightful conversations on the subject of $2$-Segal spaces. We also thank Rune Haugseng for answering some of our questions regarding the unstraightening construction and the $\inftwo$-category of bispans. Finally, we are indebted to Joachim Kock for his thorough feedback on an earlier draft of this paper.

\paragraph{Notational conventions and simplicial preliminaries.} Throughout this paper we make extensive use of the theory of $\infty$-categories as developed by Joyal \cite{JoyalJPAA, JoyalUnpub} and Lurie \cite{HTT}. In particular, by an $\infty$-category we shall always mean a {\em quasi-category}, a simplicial set having fillers for all inner horns.
 
 Ordinary, simplicially-enriched or model categories will always be denoted by either Greek letters (e.g. $\Delta$) or in ordinary font (e.g. $C$). 
 $\infty$-categories will either be blackboard Greek letters (e.g. $\Simp$) or have the first character in calligraphic font (e.g. $\cC$). 
 
 For ordinary categories $C$ and $D$, $\Fun(C,D)$ is the category of functors and natural transformations and $\Map(C,D)$ is the groupoid of functors and natural isomorphisms. Analogously, for $\infty$-categories $\cC$ and $\cD$, $\cFun(\cC,\cD)$ is the $\infty$-category of functors and $\cMap(\cC,\cD)$ is the largest Kan complex inside $\cFun(\cC,\cD)$. In particular, the ($\infty$-)category of {\em $k$-fold simplicial objects} in an ordinary category $C$ or $\infty$-category $\cC$ are 
 \begin{equation*}
  C_{\Delta^k} = \Fun\left(\left(\Delta^\op\right)^{\times k}, C\right) \ {\rm and} \ \cC_{\Simp^k} = \cFun\left(\left(\Simpop\right)^{\times k}, \cC\right),
 \end{equation*}
where $\Delta$ is the category of non-empty linearly ordered finite sets and the $\infty$-category $\Simp$ is the nerve of $\Delta$.
 
Let $\qCat$ denote the simplicially-enriched category of quasi-categories with mapping spaces given by $\cMap$. The {\em $\infty$-category of $\infty$-categories}, $\Catoo$, is the coherent nerve of $\qCat$ (\cite{HTT} 3.0.0.1). The full subcategory of $\qCat$ of Kan complexes is denoted $\Kan$, and its coherent nerve, $\cS$, is the $\infty$-category of spaces. The inclusion $\xymatrixcolsep{.8pc}\xymatrix{\cS \ar@{^{(}->}[r] &  \Catoo}$ admits a right adjoint (\cite{HTT} 1.2.5.3, 5.2.4.5) denoted
\begin{equation*}
 (-)^\simeq: \Catoo \to \cS.
\end{equation*}

We denote by $\CSS$ the category of bisimplicial sets carrying the Rezk model structure \cite{RezkCSS}. This is a simplicial model category whose fibrant-cofibrant objects are the {\em complete Segal spaces}. The coherent nerve of the full subcategory of complete Segal spaces, denoted $\cCSS$, is canonically equivalent to $\Catoo$ (\cite{JoyalTierney} 4.11).

The category of finite sets is denoted by $\Fin$ and every object is isomorphic to one of the form
 \begin{equation*}
 \underline{n}=\{1,\cdots, n\} \in \Fin.
 \end{equation*}
 Similarly, $\Finpt$ is the category of pointed finite sets. Objects of $\Finpt$ are denoted $X_*$, with $\ast$ the basepoint and $X$ the complement of the basepoint. The $\infty$-categories $\cFin$ and $\cFinpt$  are, respectively, the nerves of $\Fin$ and $\Finpt$.
 
 We adopt the topologists convention for the category $\Delta$ of non-empty, linearly ordered finite sets in that we label its objects by 
 \begin{equation*}
 [n] = \{0<\cdots<n\}\in \Delta.
 \end{equation*}
 The {\em active} and {\em inert} morphisms form a factorization system on $\Delta$. The former are those morphisms which preserve the bottom and top elements while the latter are the inclusions of subintervals. Active morphisms will be denoted by arrows of the form ($\actmorR$) while inert morphisms by arrows of the form ($\inmorR$). Their corresponding wide subcategories are denoted, respectively, by $\Deltaac$ and $\Deltain$. Every morphism in $\Delta$ can be uniquely factored as an active followed by an inert morphism. The $\infty$-categories $\Simpac$ and $\Simpin$ are, respectively, the nerves of $\Deltaac$ and $\Deltain$.
 
 \begin{remark}
  The active-inert factorisation of morphisms in $\Delta$ is a particular example of the general notion of generic-free factorisations in the theory of monads as developed by Weber \cite{Weber} and Berger-Mellies-Weber \cite{BergerEtAl}. Following Lurie \cite{LurieHA} and Haugseng \cite{RuneMorita}, we adopt the former terminology as we feel that it is more descriptive.
 \end{remark}
 
 The category $\Delta$ is a full subcategory of $\Delta_+$, the category of (possibly empty) linearly ordered finite sets, the objects of which are
 \begin{equation*}
  \langle n \rangle = \{1 < \cdots < n\} \in \Delta_+.
 \end{equation*}
 The category $\nabla$ (\cite{Kockres} 8)\footnote{Note that our category $\nabla$ is the opposite of the one defined in \cite{Kockres}} has the same objects as $\Delta_+$ and morphisms given by spans of the form 
 \begin{equation}
 \label{nabladefn}
 \xymatrixcolsep{1.1pc} \xymatrixrowsep{.8pc} \xymatrix{ & \ar[ld] \langle k \rangle \ar@{>->}[rd] & \\
 \langle n \rangle &  & \langle m \rangle\, .}
 \end{equation}
 The category of {\em $k$-fold nabla objects} in a category $C$ is
 \begin{equation*}
 C_{\nabla^k} = \Fun \left((\nabla^\op)^{\times k} , C \right).
 \end{equation*}

 There is a bijective on objects and full functor $\fG: \Delta \to \nabla$ which restricts to isomorphisms $\Deltaac \simeq \Delta_+^\op$ and $\Deltain^{\geq 1} \simeq (\Delta_+)_\mathrm{in}^{\geq 1}$(\cite{Kockres} 8.2). Restriction along $\fG$ induces a fully faithful functor
 \begin{equation}
  \label{Augdef}
  \fG^*: C_{\nabla^k} \to C_{\Delta^k}.
 \end{equation}

 The category $\Deltaac$ has a canonical monoidal structure 
 \begin{equation*}
 [n] \vee [m] = [n+m]
\end{equation*}
having unit $[0]$. The functor $\fG$ restricts to a monoidal equivalence if we endow $\Delta_+$ with the monoidal structure
\begin{equation}
\label{AugMon}
 \langle n \rangle + \langle m \rangle = \langle n+m \rangle
\end{equation}
having unit $\langle 0 \rangle$. 

Finally, throughout this paper algebra objects are assumed to be unital.

\section{Lax algebras in symmetric monoidal \texorpdfstring{$\inftwo$}{(oo,2)}-categories }
\label{Sec:LaxAlg}
In this section we cover the necessary background for defining the notion of a {\em lax algebra object} in a symmetric monoidal $\inftwo$-category. We begin in Section \ref{Sec:inftwo} with a review our chosen model of $\inftwo$-categories. In Section \ref{Sec:laxfun} we define lax functors in terms of the unstraightening construction. Finally, in Section \ref{Sec:algprop} we introduce the category $\csmAlg$ which corepresents algebras.

\subsection{Review of symmetric monoidal \texorpdfstring{$\inftwo$}{(oo,2)}-categories}
\label{Sec:inftwo}
The model of $\inftwo$-categories that we use in this work is the one originally introduced by Lurie in \cite{Lurieinftwo}. 

\begin{defn}
 An {\em $\inftwo$-category} is a simplicial $\infty$-category $\cB \in (\Catoo)_\Simp$ such that
 \begin{enumerate}
  \item $\cB$ is a Segal object, that is, for each $n\geq 2$ the functor $\cB_n \to \cB_1 \times_{\cB_0}\cdots \times_{\cB_0} \cB_1$ is an equivalence;
  \item The $\infty$-category $\cB_0$ is a space, that is, $\cB_0 \in \cS \subset \Catoo$; and,
  \item The Segal space $\xymatrixcolsep{.8pc}\xymatrix{ \Simpop \ar[r]^-\cB & \Catoo \ar[r]^-{(-)^\simeq} & \cS}$ is complete.
 \end{enumerate}
\end{defn}
The $\infty$-category of $\inftwo$-categories is a full subcategory $\xymatrixcolsep{.8pc}\xymatrix{ \Bicatoo \ar@{^{(}->}[r] & (\Catoo)_\Simp}$. In other words, functors between $\inftwo$-categories are simply natural transformations.
 
 \begin{remark}
A widely used model of $\inftwo$-categories in the literature are $2$-fold complete Segal spaces as introduced by Barwick \cite{BarwickThesis}. Barwick's model is recovered from the one used in this work by presenting $\Catoo$ as $\cCSS$.  
 \end{remark}

Let $\op: \Delta \to \Delta$ denote the automorphism sending $[n]$ to $[n]^\op$.
\begin{defn}
 For an $\inftwo$-category $\cB$ its {\em opposite $\inftwo$-category} is the composite
 \begin{equation*}
  \cB^\op: \xymatrixcolsep{1.5pc} \xymatrix{\Simpop \ar[r]^-\op & \Simpop \ar[r]^-\cB & \Catoo}.
 \end{equation*}
\end{defn}

 There is a universal way to extract from a Segal object $\cB$ in $\Catoo$ a new Segal object $\cU \cB$ whose $\infty$-category of $0$-simplices is a space: let $\cSeg \Catoo$ denote the full subcategory of $(\Catoo)_\Simp$ spanned by Segal objects and let $\cSego \Catoo$ be the full subcategory of the former spanned by those objects satisfying Condition $2$ above. Then the inclusion $\xymatrixcolsep{.8pc}\xymatrix{\cSego \Catoo \ar@{^{(}->}[r] & \cSeg \Catoo}$ admits a right adjoint (\cite{RuneSpans} 2.13)
 \begin{equation*}
 \cU:\cSeg \Catoo \to \cSego \Catoo. 
 \end{equation*}
 Explicitly, given $\cB \in \cSeg \Catoo$, we have $\cU \cB_0 = \cB_0^{\simeq}$ and for each $n\geq 1$ a pullback square
 \begin{equation}
 \label{LocPB}
  \xymatrixrowsep{.9pc}\xymatrixcolsep{.9pc}\xymatrix{\cU \cB_n \ar[r] \ar[d] & \cB_n \ar[d] \\
  \left(\cB_0^\simeq\right)^{n+1} \ar@{^{(}->}[r] & \left( \cB_0\right)^{n+1} }
 \end{equation}

There are a number of ways in the literature to define symmetric monoidal $\inftwo$-categories. We follow Lurie (\cite{LurieHA} 2.0.0.7) in choosing the one generalising Segal's notion of a special $\Gamma$-space \cite{SegalGamma}. 
\begin{defn}
 A {\em symmetric monoidal $\inftwo$-category} is a functor 
 \begin{equation*}
 \cB^\otimes :\cFinDel\to \Catoo
 \end{equation*}
 such that
 \begin{enumerate}
  \item For each $S_* \in \cFinpt$, the simplicial $\infty$-category $\cB^\otimes(S_*,\bullet)$ is an $\inftwo$-category.
  \item For each $[n] \in \Simpop$ and $S_* \in \cFinpt$, the map $\cB^\otimes(S_*,[n]) \to \prod_{s \in S} \cB^\otimes(\{s\}_*, [n])$ is an equivalence.
 \end{enumerate}
\end{defn}
 The $\infty$-category of symmetric monoidal $\inftwo$-categories is a full subcategory 
 \begin{equation*}
 \xymatrixcolsep{.8pc}\xymatrix{ \Bicatoo^\otimes \ar@{^{(}->}[r] & \cFun\left(\cFinDel, \Catoo \right)}.
 \end{equation*}
 In other words, a symmetric monoidal functor between symmetric monoidal $\inftwo$-categories is simply a natural transformation, that is, a morphism in $\cFun\left(\cFinDel, \Catoo \right)$. 

\subsection{Symmetric monoidal lax functors via the unstraightening construction}
\label{Sec:laxfun}

Recall that a lax functor between ordinary $2$-categories $L: A \rightsquigarrow B$ differs from a functor in that it no longer respects identity arrows and composition of $1$-morphisms \cite{LackComp}. Instead, for each object $a$ of $A$ and each pair of composable morphisms $f$ and $g$, one has (not necessarily invertible) $2$-morphisms $\id_{L(a)} \Rightarrow L(\id_a)$ and $L(g) \circ L(f) \Rightarrow L(g \circ f)$ in $B$ witnessing the lax preservation of unitality and composition. Furthermore, these $2$-morphisms must satisfy associativity and unitality coherence equations.

The notion of a lax functor between $\inftwo$-categories requires the use of the theory of {\em cocartesian fibrations} of $\infty$-categories and the {\em unstraightening construction} as developed by Lurie (\cite{HTT} 2.4). Recall that the unstraightening construction defines an equivalence between functors $\cC \to \Catoo$ and cocartesian fibrations over $\cC$ (\cite{HTT} 3.2.0.1),
\begin{equation*}
\Un: \xymatrixcolsep{.8pc}\xymatrix{\cFun(\cC, \Catoo) \ar[r]^-\sim & {\cal C}{\rm ocart}_{/\cC}}.
\end{equation*}
Under this equivalence a natural transformation $\eta:F \Rightarrow G$ is sent to a morphism
\begin{equation*}
 \xymatrixrowsep{.8pc}\xymatrixcolsep{.9pc}\xymatrix{\Un(F)\ar[rr]^{\Un(\eta)} \ar[rd] & & \Un(G) \ar[ld] \\
  & \cC & }
\end{equation*}
such that $\Un(\eta)$ sends cocartesian morphisms in $\Un(F)$ to cocartesian morphisms in $\Un(G)$. The unstraightening construction is natural in $\cC$ in the sense that from a composite
\begin{equation*}
 \xymatrixcolsep{1.4pc} \xymatrix{ \cD \ar[r]^-G & \cC \ar[r]^-F & \Catoo}
\end{equation*}
one has a pullback diagram (\cite{GHN} A.31)
\begin{equation*}
 \xymatrixrowsep{.9pc} \xymatrixcolsep{1.4pc} \xymatrix{ \Un(FG) \ar[d] \ar[r] & \Un(F) \ar[d] \\
 \cD \ar[r]_-{G} & \cC\,. }
\end{equation*}

Taking $\cC=\Simpop$ one can unstraighten an $\inftwo$-category, and functors become morphisms of cocartesian fibrations over $\Simpop$ which preserve cocartesian morphisms. Lax functors will still be morphisms of fibrations but will preserve fewer cocartesian morphisms. 
\begin{defn}
 A {\em lax functor} $L:\cA \rightsquigarrow \cB$ between $\inftwo$-categories $\cA, \, \cB: \Simpop \to \Catoo$ is a morphism
 \begin{equation*}
  \xymatrixrowsep{.8pc}\xymatrixcolsep{.9pc}\xymatrix{\Un\left(\cA\right)\ar[rr]^{L} \ar[rd] & & \Un\left(\cB\right) \ar[ld] \\
  & \Simpop & }
 \end{equation*}
 such that $L$ sends cocartesian lifts of morphisms in $(\Simpin)^{\rm op}$, the subcategory of inert morphisms, to cocartesian morphisms.
\end{defn}
\begin{remark}
We are not certain as to the exact history of this approach to defining lax functors between $\inftwo$-categories. We first learned it from Dyckerhoff--Kapranov (\cite{DK12} 9.2.8) and Lurie's definition of a morphism of $\infty$-operads (\cite{HTT} 2.1.2.7) is qualitatively similar. We believe an analogous statement must be known for lax functors between $2$-categories but do not know any references.
 \end{remark}

\begin{remark}
 \label{laxop}
 Since $\op:\Delta \to \Delta$ preserves inert morphisms, the naturality of the unstraightening construction implies that from a lax functor $L: \cA \rightsquigarrow \cB$ one gets a lax functor $L^\op: \cA^\op \rightsquigarrow \cB^\op$.
\end{remark}

It is straightforward to extend this notion to define symmetric monoidal lax functors, that is, functors between symmetric monoidal $\inftwo$-categories which preserve the symmetric monoidal structure but only laxly preserve composition, as follows.
\begin{defn}
\label{DefnsmLax}
 A {\em symmetric monoidal lax functor} $L: \cA^\otimes \rightsquigarrow \cB^\otimes$ between symmetric monoidal $\inftwo$-categories $\cA^\otimes, \cB^\otimes: \cFinDel \to \Catoo$ is a morphism
 \begin{equation*}
  \xymatrixrowsep{.8pc}\xymatrixcolsep{.9pc}\xymatrix{\Un\left(\cA^\otimes\right)\ar[rr]^{L} \ar[rd] & & \Un\left(\cB^\otimes\right) \ar[ld] \\
  & \cFinDel & }
 \end{equation*}
 such that $L$ sends cocartesian lifts of morphisms in $\cFinpt\times(\Simpin)^{\rm op}$ to cocartesian morphisms.
\end{defn}

\subsection{Corepresenting algebra objects}
\label{Sec:algprop}
To specify an algebra object in a symmetric monoidal $\inftwo$-category one must provide not just an associative and unital binary operation, but a coherent choice of higher associativity and unitality data. To package together all of this data we will make use of a category originally introduced by Pirashvili \cite{Pirash02}.

Denote by $\Alg$ the category having as objects finite sets. A morphism in $\Alg$ is a function $p:X \to Y$ along with a choice of linear ordering of the (possibly empty) preimages $p^{-1}(y)$ for each $y \in Y$. The composition of a pair of composable morphisms
\begin{equation*}
 \xymatrix{ X_0 \ar[r]^-{p_1} & X_1 \ar[r]^-{p_2} & X_2}
\end{equation*}
is the composition of the underlying functions, with linear ordering on $(p_2p_1)^{-1}(x_2)$ given by
\begin{equation*}
 (p_2p_1)^{-1}(x_2) = \sum_{x_1 \in p_2^{-1}(x_2)} p_1^{-1}(x_1),
\end{equation*}
where the sum denotes the monoidal structure on $\Delta_+$, the category of finite linear orders, introduced in Eq.~\ref{AugMon}. The disjoint union endows $\Alg$ with a symmetric monoidal structure.

The category $\Alg$ corepresents algebra objects in the sense that, for a symmetric monoidal category $C$, symmetric monoidal functors $\Alg\to C$ are the same as algebra objects in $C$. This is because $\Alg$ is the {\em category of operators} \cite{MayOps} for the $\Sigma$-operad of associative algebras, that is, ${\rm Hom}_\Alg(\underline{n},\underline{1}) \simeq \Sigma_n$ and all morphisms in $\Alg$ are, up to precomposition with an isomorphism, disjoint unions of these. 

\begin{remark}
 One only needs a monoidal structure on a category to define algebra objects in it. The simpler category $\Delta_+$ corepresents algebra objects in {\em monoidal} categories and so one can equivalently define an algebra object in a symmetric monoidal category $C$ to be a {\em monoidal} functor from $\Delta_+$ to $C$. However, one needs at least a {\em braiding} on a monoidal category to define bialgebra objects in it, and in all the examples which concern us the braiding is in fact symmetric. As this paper lays the foundations for our work on higher categorical bialgebras \cite{penney2017bialg, penney2017bimon} it is therefore crucial that we make use of $\Alg$ rather than $\Delta_+$.
\end{remark}

To corepresent algebra objects in a symmetric monoidal $\inftwo$-category it suffices to present $\Alg$ in the model of symmetric monoidal $\inftwo$-categories that we use in this paper. Observe that for $\cC$ an $\infty$-category and $F:\cC_1 \to \cC_2$ a fully faithful functor, both the right and left Kan extensions,
 \begin{equation*}
  F_*,F_!: \cFun(\cC_1, \cC) \to \cFun(\cC_2, \cC),
 \end{equation*}
are fully faithful should they exist.  
\begin{defn}
\label{RKERem}
 Let $\cC$ be an $\infty$-category and $F:\cC_1 \to \cC_2$ be a fully faithful functor. We call functors in the image of $F_*$ {\em cartesian} and those in the image of $F_!$ {\em cocartesian}. We denote these full subcategories, respectively, by $\cFun^{\rm cart}(\cC_2, \cC)$ and $\cFun^{\rm cocart}(\cC_2,\cC)$, with the functor $F$ to be understood implicitly from the context.
\end{defn}
\begin{remark}
We follow Haugseng's terminology \cite{RuneSpans} as our definition of a (co)cartesian functor is a generalisation of the one given there. In Section \ref{Sec:laxfun} we discussed cocartesian morphisms and cocartesian fibrations. These are distinct notions from the one being introduced now, but as these terms are used in different contexts there is little fear of confusion.
\end{remark}

 For each set $S$, denote by $\Cart S$ the poset of subsets of $S$ ordered by inclusion. These assemble into a functor $\Cart -: \Finpt^{\rm op} \to \Cat$ by declaring the image of a pointed map $f:S_* \to T_*$ to be
 \begin{equation*}
  \Cart f: \Cart T \to \Cart S, \quad U \mapsto f^{-1}(U).
 \end{equation*}
For each $S_* \in \Finpt$ there is a full subcategory $\xymatrixcolsep{1.5pc}\xymatrix{\Sing S \ar@{^{(}->}[r] & \Cart S}$ consisting of the singleton subsets. The $\infty$-categories $\cSing S$ and $\cCart S$ are, respectively, the nerves of $\Sing S$ and $\Cart S$.

\begin{exam}
A functor $\cCartop {\{1,2\}}\to \cC$ is a diagram,
\begin{equation*}
 \xymatrixrowsep{.8pc}\xymatrix{ & c_{\{1,2\}} \ar[rd] \ar[ld] & \\
 c_{\{1\}} \ar[rd] & &\ar[ld] c_{\{2\}} \\
  & c_{\emptyset} & }
\end{equation*}
Such a diagram is cartesian if it presents $c_{\{1,2\}}$ as the product of $c_{\{1\}}$ and $c_{\{2\}}$ and $c_{\emptyset}$ is terminal. Similarly, a cartesian functor $\cCartop S\to \cC$ encodes a coherent choice of products for a collection of objects of $\cC$ labelled by the elements of $S$.
\end{exam}

Recall that the symmetric monoidal structure on an $\infty$-category $\cD$ having finite coproducts is given by the functor
\begin{equation*}
 \cD^\amalg: \cFinpt \to \Catoo, \quad S_* \mapsto \cFun^{\rm cocart}\left( \cCart S, \cD \right).
\end{equation*}
While the disjoint union is {\em not} the coproduct in $\Alg$, it is the case that given morphisms $p_i:X_i \to Y_i$ there is a unique morphism making the following diagram commute
\begin{equation*}
 \xymatrixrowsep{.8pc}\xymatrixcolsep{.9pc}\xymatrix{X_1 \ar[rrr]^{p_1} \ar@{_{(}->}[rd] & &  & Y_1 \ar@{^{(}->}[ld]\\
  & X_1 \coprod X_2 \ar@{-->}[r] & Y_1 \coprod Y_2 & \\
  X_2 \ar[rrr]^{p_2} \ar@{^{(}->}[ru] & &  & Y_2 \ar@{_{(}->}[lu]} 
\end{equation*}
Define a functor $F:\Cart S \times [n] \to \Alg$ to be {\em cocartesian} if for each $i \in [n]$, the composite
\begin{equation*}
 \xymatrix{\Cart S \ar[r]^{F(-,i)} & \Alg \ar[r]^{\rm forget} & \Fin}
\end{equation*}
is cocartesian. One has by the above that the functor 
\begin{equation*}
\csmAlg: \FinDel \to \qCat^0
\end{equation*}
sending $(S_*,[n])$ to the set of cocartesian functors $\Cart S \times [n] \to \Alg$ is a symmetric monoidal $\inftwo$-category.

We can now define the algebraic structures which are the focus of the remainder of this paper.
\begin{defn}
\label{DefnLaxAlg} \label{DefnLaxMAlg}
Let $\cB^\otimes$ be a symmetric monoidal $\inftwo$-category.
\begin{itemize}
 \item An {\em algebra object} in $\cB^\otimes$ is a symmetric monoidal functor $\csmAlg \to \cB^\otimes$.
  \item A {\em lax algebra object} in $\cB^\otimes$ is a symmetric monoidal lax functor $\csmAlg \rightsquigarrow \cB^\otimes$.
\end{itemize}
\end{defn}
\begin{remark}
 \label{CoalgRem}
 One can readily dualise the preceding discussion to define coalgebra objects in symmetric monoidal $\inftwo$-categories. Namely, a {\em (lax) coalgebra object} in $\cB^\otimes$ is a symmetric monoidal (lax) functor from $\csmCoalg := (\csmAlg)^\op$ to $\cB^\otimes$.
\end{remark}

It is worth taking a moment to informally discuss the exact nature of a lax algebra object $A$, as it is slightly more subtle than one might initially expect. For each string of composable morphisms in $\Alg$,
\begin{equation*}
 \xymatrixcolsep{1.5pc} \xymatrix{ \underline{n}_0 \ar[r]^-{p_1} & \underline{n}_1 \ar[r]^-{p_2} &  \cdots \ar[r]^-{p_k} & \underline{n}_k}
\end{equation*}
one has $2$-morphisms
\begin{equation*}
 \xymatrixrowsep{.9pc} \xymatrixcolsep{1.3pc} \xymatrix{ & A^{\otimes n_1} \ar[r]^-{p_2} & \cdots \ar[r]^-{p_{k-1}} \ar@{=>}[d] & A^{\otimes n_{k-1}} \ar[rd]^-{p_k} & \\
 A^{\otimes n_0}\ar[ru]^-{p_1} \ar[rrrr]_-{p_k \circ \cdots \circ p_1} & & & & A^{\otimes n_k} }
\end{equation*}
which are compatible with disjoint union and the composition of morphisms in $\Alg$. In particular, the witness to the lax associativity of the product on $A$ is a diagram
\begin{equation*}
 \xymatrixrowsep{.55pc}\xymatrixcolsep{.55pc} \xymatrix{ A^{\otimes 3} \ar[rr]^-{\mu \otimes \id} \ar[dd]_-{\id \otimes \mu} \ar[rrdd] & & A^{\otimes 2} \ar[dd]^-{\mu} \ar@{=>}[ld] \\
   & &  \\
   A^{\otimes 2} \ar@{=>}[ru] \ar[rr]_-{\mu} & & A\, .}
\end{equation*}

\section{The \texorpdfstring{$\inftwo$}{(oo,2)}-category of bispans}
\label{Sec:bispans}

The construction which associates the $\infty$-category $\cSpan \cC$ to an $\infty$-category $\cC$ having finite limits can be iterated to form $(\infty,n)$-categories for each $n$. Informally, a $2$-morphism between spans is a `span of spans', that is a diagram of the form
\begin{equation*}
 \xymatrixrowsep{.8pc}\xymatrix{ & d \ar[rd] \ar[ld] & \\
 c & e \ar[u] \ar[d] \ar[r] \ar[l]& c' \\
 & d' \ar[ru] \ar[lu] & }
\end{equation*}
while a $3$-morphism is a `span of spans of spans', and so on. The cartesian product in $\cC$ endows these $(\infty,n)$-categories with a symmetric monoidal structure. The rigorous construction of these symmetric monoidal $(\infty,n)$-categories has been carried out by Haugseng \cite{RuneSpans}. 

Section \ref{Sec:twisted} is a brief discussion on the twisted arrow construction, a construction which appears throughout this work. In Section \ref{BispanConst} we review Haugseng's construction in the case that concerns us, namely the construction of the symmetric monoidal $\inftwo$-category $\csmtSpan \cC$ of {\em bispans} in $\cC$. In Section \ref{StrictBispan} we prove that $\csmtSpan \cC$ is {\em semistrict}, a technical condition which simplifies the description of lax functors. Finally, in Section \ref{UnstrSpan} we determine explicitly the unstraightening of $\csmtSpan{N(C)}$ for $C$ an ordinary category having finite limits.

\subsection{The twisted arrow category}
\label{Sec:twisted}

The {\em twisted arrow category} (\cite{MacLane} IX.6.3), $\Pyr D$, of an ordinary category $D$ will be a recurring character in this work. It will first appear in the definition of the $\inftwo$-category of bispans in Section \ref{BispanConst} and its subsequent reappearances will be tied to various constructions involving bispans. 

For an ordinary category $D$, let $\Pyr D$ be the category having as objects arrows $f:d \to d'$ in $D$, and morphisms from $f_1$ to $f_2$ 
\begin{equation*}
 \xymatrixrowsep{.9pc} \xymatrix{d_1 \ar[r]^-{f_1} \ar[d] & d_1' \\
 d_2 \ar[r]_-{f_2} & d_2' \ar[u] }
\end{equation*}
Remembering only the source and target of an object of $\Pyr D$ defines a forgetful functor $\Pyr D \to D \times D^\op$. Furthermore, the twisted arrow categories assemble into a functor $\Pyr {-}:\Cat \to \Cat$.

\begin{remark}
 There are two equally canonical conventions for the definition of the twisted arrow category, the second of which has morphisms given by diagrams
 \begin{equation*}
 \xymatrixrowsep{.9pc} \xymatrix{d_1 \ar[r]^-{f_1}  & d_1' \ar[d] \\
 d_2 \ar[r]_-{f_2}\ar[u]  & d_2'  }
\end{equation*}
We follow the convention used by Haugseng \cite{RuneSpans}, while the second convention is used by Barwick \cite{BarwickSpan} and Lurie (\cite{LurieHA} 5.2.1).
\end{remark}

A number of constructions in later sections involve writing explicit functors of the form $\Pyr D \to C$ for $C$ a category having finite limits. It turns out that such functors can be equivalently described as {\em normal oplax functors} $D \nrightarrow \sp C$, where $\sp C$ is a bicategory which we shall describe shortly. For our purposes this latter description will often be more convenient.

Given a category $C$ having finite limits one can define a bicategory $\sp C$ \cite{Benabou} having the same objects as $C$, $1$-morphisms given by spans, and $2$-morphisms given by diagrams
\begin{equation*}
 \xymatrixrowsep{.7pc} \xymatrix{ & d \ar[rd] \ar[ld] \ar[dd] & \\
 c & & c' \, .\\
 & d' \ar[ru] \ar[lu] & }
\end{equation*}
Horizontal composition is given by pullbacks, chosen once and for all. The identity $1$-morphism for an object $c$ is the span $\xymatrixcolsep{.7pc} \xymatrix{ c \ar@{=}[r] & c \ar@{=}[r] & c}$.
\begin{remark}
Note that this bicategory is similar to, but distinct from, the $\inftwo$-category of bispans in $C$ that we shall introduce in Section \ref{BispanConst}. They have different $2$-morphisms and, unlike the $\inftwo$-category of bispans in $C$, $\sp C$ has no $k$-morphisms for $k>2$.
\end{remark}
A {\em normal oplax functor} $F: D \nrightarrow \sp C$ consists of, for each object $d \in D$ an object $F(d) \in C$, for each morphism $f:c \to d$ in $D$ a span
\begin{equation*}
 \xymatrixrowsep{.8pc} \xymatrix{ & \ar[ld] F(f) \ar[rd] &  \\
 F(c) & & F(d) }
\end{equation*}
and for each pair of composable morphisms $\xymatrixcolsep{.7pc} \xymatrix{ c \ar[r]^-f & d \ar[r]^-g & e}$ a diagram
\begin{equation*}
 \xymatrixrowsep{.7pc} \xymatrix{ & F(g\circ f) \ar[rd] \ar[ld] \ar[dd]^-{\Phi_{g,f}} & \\
 F(c) & & F(e) \, .\\
 & F(g)\times_{F(d)} F(f) \ar[ru] \ar[lu] & }
\end{equation*}
For each object $d \in D$ it must be that $F(\id_d)=\id_{F(d)}$, for each morphism $\xymatrixcolsep{.7pc} \xymatrix{c \ar[r]^-f & d}$ it must be that $\Phi_{\id_d, f} = \id_{F(f)}=\Phi_{f, \id_c}$, and, suppressing associator isomorphisms, for each string of composable morphisms $\xymatrixcolsep{.7pc} \xymatrix{ b \ar[r]^-f & c \ar[r]^-g & d \ar[r]^-h & e}$,
\begin{equation}
\label{OplaxAssoc}
 \left(\id_{F(h)} \times_{F(d)} \Phi_{g,f}\right) \circ \Phi_{h,gf} = \left( \Phi_{h,g} \times_{F(c)} \id_{F(f)}\right) \circ \Phi_{hg,f},
\end{equation}
as morphisms $F(hgf) \to F(h) \times_{F(d)} F(g) \times_{F(c)} F(f)$. 

\begin{thm}[\cite{Errington} 3.4.1]
\label{ErrOplax}
 For any category $D$ and any category $C$ having finite limits there is a natural bijection
 \begin{equation*}
  \mathrm{Hom}\left(\Pyr{D},C\right) \simeq \mathrm{Hom}^{\rm n.oplax} \left(D, \sp C \right),
 \end{equation*}
 where $\mathrm{Hom}^{\rm n.oplax}$ denotes the set of normal oplax functors.
\end{thm}
The isomorphism is given as follows. A normal oplax functor $F:D \nrightarrow \sp C$ associates to each diagram 
\begin{equation*}
 \xymatrixrowsep{.9pc} \xymatrix{d_1 \ar[r]^-{f_1} \ar[d]_-g & d_1' \\
 d_2 \ar[r]_-{f_2} & d_2'\ .\ar[u]_-h }
\end{equation*}
a diagram
\begin{equation*}
 \xymatrixrowsep{.9pc} \xymatrixcolsep{1pc} \xymatrix{ & & & F(f_1) \ar[rrd] \ar[d]_-{F_{g,h}} \ar[lld] & & & \\
  & F(g) \ar[ld] \ar[rd] & & F(f_2) \ar[ld] \ar[rd] & & F(h) \ar[ld] \ar[rd] & \\
  F(d_1) & & F(d_2) & & F(d_2') & & F(d_1')\, .}
\end{equation*}
The corresponding functor $\tilde{F}: \Pyr D \to C$ is then
\begin{equation*}
 \tilde{F}: f \mapsto F(f), \quad \left((g,h):f_1 \to f_2\right) \mapsto F_{g,h}.
\end{equation*}

Finally, an immediate corollary of Theorem \ref{ErrOplax} is the following:
\begin{cor}
 \label{Twistedcolim}
 Let $\Cat^\mathrm{lex}$ denote the category of categories having finite limits and finite limit preserving functors between them. Then for any diagram $L:J \to \Cat$, the functors
 \begin{equation*}
  \mathrm{Hom}\left( \Pyr{\colim L}, - \right), \mathrm{Hom}\left( \colim_{j \in J} \Pyr{L(j)}, -\right):\Cat^\mathrm{lex} \to \Set
 \end{equation*}
are naturally isomorphic.
\end{cor}
\begin{proof}
 Since bicategories and normal oplax functors between them assemble into a category, the set-valued functor $\mathrm{Hom}^{\rm n.oplax}(-,-)$ sends colimits in the first variable to limits in sets. Let $L:J \to \Cat$ be a diagram in $\Cat$. By Theorem \ref{ErrOplax} one has for each $C$ having finite limits the following string of natural bijections
 \begin{align*}
  \mathrm{Hom}\left(\Pyr{\colim L}, C \right) &\simeq \mathrm{Hom}^{\rm n.oplax} \left(\colim L , \sp C \right) \simeq \lim_{j \in J} \mathrm{Hom}^{\rm n.oplax}\left(L(j), \sp C \right) \\
  &\simeq \lim_{j \in J} \mathrm{Hom} \left( \Pyr{L(j)}, C \right) \simeq \mathrm{Hom} \left( \colim_{j \in J} \Pyr{L(j)}, C\right).
 \end{align*}

\end{proof}

\subsection{Definition of \texorpdfstring{$\csmtSpan \cC$}{Span(C)}}
\label{BispanConst}
Haugseng's construction of the symmetric monoidal $\inftwo$-category $\csmtSpan \cC$ is an iteration of Barwick's construction of the $\infty$-category $\cSpan \cC$ \cite{BarwickSpan}. One first defines a symmetric monoidal double $\infty$-category $\overline{\csmtSpan \cC} \in \cSeg {\Catoo^\otimes}$ which fails to be in $\cSego {\Catoo^\otimes}$. One then remedies this problem by defining $\csmtSpan \cC := \cU \left(\overline{\csmtSpan \cC}\right)$.

The morphisms in the $\inftwo$-category of bispans are given by diagrams in $\cC$ of the following form. Let $\Pyr{n}:=\Pyr{[n]}$ be the twisted arrow construction of $[n]$. Explicitly, this is the opposite of the poset of non-empty intervals in $[n]$. Similarly, the category $\Pyr{n,k} := \Pyr{[n] \times [k]} = \Pyr{n} \times \Pyr k$ is opposite of the poset of non-empty rectangles in $[n] \times [k]$. These assemble into a functor 
\begin{equation*}
\Pyr {\bullet,\bullet}: (\Delta)^2 \to \Cat. 
\end{equation*}
There is a full subcategory $\xymatrixcolsep{.95pc}\xymatrix{\Wedge n \ar@{^{(}->}[r] & \Pyr n}$ consisting of intervals $[i;j]$ with $|j-i| \leq 1$ and hence a full subcategory  $\xymatrixcolsep{1.1pc}\xymatrix{\Wedge {n,k} \ar@{^{(}->}[r] & \Pyr {n,k}}$. The $\infty$-categories $\cWedge {n}$ and $\cPyr {n}$ are, respectively, the nerves of $\Wedge {n}$ and $\Pyr {n}$.

\begin{exam}
\label{Sigmaintuit}
 A functor $\cPyr 2 \to \cC$ is a diagram,
\begin{equation*}
  \xymatrixrowsep{.8pc}\xymatrix{ & & c_{[0;2]} \ar[rd] \ar[ld] & & \\
  & c_{[0;1]} \ar[rd] \ar[ld] & & c_{[1;2]} \ar[rd] \ar[ld] & \\
  c_{[0;0]} & & c_{[1;1]} & & c_{[2;2]} \, . }
 \end{equation*}
Such a diagram is cartesian if it is the right Kan extension of its restriction to $\cWedge 2$, i.e., if the middle square is a pullback in $\cC$. In general, a functor $\cPyr n \to \cC$ is pyramid of spans on $n+1$ objects. It being cartesian says higher tiers of this pyramid consist of a coherent choice of pullbacks of the $n$ spans along the bottom two tiers. 
 \end{exam}

Presenting $\Catoo$ as $\cCSS$ we can now define the symmetric monoidal $\inftwo$-category of bispans. For an $\infty$-category $\cC$ having finite limits, consider the functor
\begin{equation*}
 \Finpt\times \left(\Delta^{\rm op}\right)^2 \to \Kan, \quad (S_*, [n],[k]) \mapsto \cMap^{\rm cart}\left(\cCartop S \times \cPyr {n,k}, \cC \right),
\end{equation*}
which is well-defined by Proposition 3.8 of \cite{RuneSpans}. Taking the coherent nerve defines a functor $\overline{\csmtSpan \cC}:\cFinpt \times \left(\Simpop\right)^2 \to \cS$.

\begin{prop}
 Let $\cC$ be an $\infty$-category with finite limits. Then the functor 
 \begin{equation*}
 \csmtSpan \cC:=\cU \left(\overline{\csmtSpan \cC}\right): \cFinDel \to \Catoo
 \end{equation*}
 is a symmetric monoidal $\inftwo$-category. We call this the {\em symmetric monoidal $\inftwo$-category of bispans in $\cC$}.
\end{prop}
\begin{remark}
 Our description of the symmetric monoidal structure arising from the cartesian product on $\cC$ differs from Haugseng. He instead presents the symmetric monoidal structure by giving a sequence of $(\infty,2+k)$-categories delooping the $\inftwo$-category of bispans (\cite{RuneSpans} 9.1).
\end{remark}

\begin{proof}
 By Corollary 6.5 of \cite{RuneSpans}, for each fixed $S_* \in \cFinpt$, the simplicial $\infty$-category $\csmtSpan {\cC}(S_*, \bullet, \bullet)$ is a $\inftwo$-category. Consider the following commutative diagram in $\cS$,
 \begin{equation*}
  \xymatrixrowsep{.8pc}\xymatrix{\overline{\csmtSpan {\cC}}(S_*,[n],[k]) \ar[d]_-{\wr} \ar[r] & \prod_{s \in S} \overline{\csmtSpan {\cC}}(\{s\}_*,[n],[k]) \ar[d]^-{\wr} \\
  \cMap^{\rm cart}\left(\cCartop S \times \cPyr {n,k}, \cC \right) \ar[r]_{j^*} & \cMap\left(\cSingop S \times \cWedge {n,k}, \cC \right) \, ,}
 \end{equation*}
 where $j$ denotes the fully faithful functor including $\cSingop S \times \cWedge{n,k}$ into $\cCartop S \times \cPyr{n,k}$. Since $\cU$ preserves limits it suffices to show that the top map is an equivalence. 
 
 By Definition \ref{RKERem}, the category $\cFun^{\rm cart}\left(\cCartop S \times \cPyr {n,k}, \cC \right)$ is the image of $\cFun\left(\cSingop S \times \cWedge {n,k}, \cC \right)$ under $j_*$, the right adjoint of $j^*$. Since $j_*$ is fully faithful it is an equivalence onto its image with inverse $j^*$. It follows that the top map in the above diagram is an equivalence, as desired.
\end{proof}

\subsection{\texorpdfstring{$\csmtSpan \cC$}{Span(C)} is semistrict}
\label{StrictBispan}
Describing symmetric monoidal lax functors for general symmetric monoidal $\inftwo$-categories can be quite complicated due to the use of the unstraightening construction. For the following class of symmetric monoidal $\inftwo$-categories it simplifies greatly.

Recall that a functor $N(C) \to \Catoo$ is equivalent to a functor $\fC N(C)\to \qCat$ of simplicially enriched categories, where $\fC$ is the left adjoint of the coherent nerve (\cite{HTT} 1.1.5).
\begin{defn}
For an ordinary category $C$, a functor $N(C) \to \Catoo$ is called {\em semistrict} if there is a functor of ordinary categories $C \to \qCat^0$, for $\qCat^0$ the ordinary category of quasi-categories and functors between them, such that the following commutes
\begin{equation*}
 \xymatrixrowsep{.9pc} \xymatrix{ \fC N(C) \ar[d]_-{\epsilon} \ar[r] & \qCat \\
 C \ar[r] & \qCat^0 \ar@{^{(}->}[u]\, , }
\end{equation*}
where $\epsilon$ is the counit of the adjunction $\fC \dashv N$. In particular, a symmetric monoidal $\inftwo$-category is semistrict if it is semistrict as a functor $\cFinDel \to \Catoo$.
\end{defn}

The aim of this section is to show that $\csmtSpan \cC$ is semistrict. To that end, we must first compute the pullback of $\infty$-categories in diagram \ref{LocPB}, or equivalently, the homotopy pullback of
\begin{equation}
\label{LocPBCSS}
 \xymatrixrowsep{.8pc}\xymatrix{   & \overline{\csmtSpan \cC}(S_*,[n], \bullet) \ar[d]^-v \\
 \left(\left(\overline{\csmtSpan \cC}(S_*,[0],\bullet)\right)^\simeq\right)^{n+1} \ar@{^{(}->}[r] & \left(\overline{\csmtSpan \cC}(S_*,[0],\bullet)\right)^{n+1} }
\end{equation}
in $\CSS$, the category of bisimplicial sets endowed with the Rezk model structure \cite{RezkCSS}.

The computation is rendered trivial by the following lemma.
\begin{lem}
\label{FibProp}
 For each $(S_*,[n]) \in \FinDel$, the morphism 
 \begin{equation*}
 v:\overline{\csmtSpan \cC}(S_*,[n], \bullet) \to \left(\overline{\csmtSpan \cC}(S_*,[0],\bullet)\right)^{n+1}
 \end{equation*}
 is a fibration in $\CSS$.
\end{lem}
\begin{proof}
 It suffices to show that $v$ is a Reedy fibration since $\CSS$ is a left Bousfield localisation of the Reedy model structure on the category of bisimplicial sets, and both the source and the target of $v$ are fibrant in $\CSS$ (\cite{Hirschhorn} 3.3.16).
 
 For the purposes of this proof we will use the shorthand $\cM(-,-) := \cMap(-,-)$. Consider the following commutative diagram in $\sSet$,
 \begin{equation}
 \label{FibPropPB}
  \xymatrix{\cM^{\rm cart} \left(\cCartop S \times \cPyr {n,k}, \cC\right) \ar[r] \ar[d] & \left( \cM^{\rm cart} \left( \cCartop S \times \cPyr {n,\bullet}, \cC\right) \right)^{\partial \Simplex k} \ar[d] \\
  \left( \cM^{\rm cart} \left(\cCartop S \times \cPyr {0,k}, \cC\right) \right)^{n+1} \ar[r] & \left(\left( \cM^{\rm cart} \left(\cCartop S \times \cPyr {0,\bullet}, \cC\right) \right)^{n+1}\right)^{\partial \Simplex k}}
 \end{equation}
 The morphism $v$ is a Reedy fibration if and only if the induced map $\overline{\nu}$ from $\cM^{\rm cart} \left(\cCartop S \times \cPyr {n,k}, \cC\right)$ to the pullback is a Kan fibration of simplicial sets (\cite{RezkCSS} 2.4).
 
 Let $B(n,k)$ be the pushout of simplicial sets
\begin{equation*}
 \xymatrix{(\Simplex 0)^{\amalg n+1} \times \partial \Simplex k \ar[r] \ar[d] & \Simplex n \times \partial \Simplex k \ar[d] \\
 (\Simplex 0)^{\amalg n+1} \times \Simplex k \ar[r] & B(n,k) \, .}
\end{equation*}
The pullback of diagram \ref{FibPropPB} is $\lim_{\Simplex m \to B(n,k)} \cM^{\rm cart}\left( \cCartop S \times \cPyr m , \cC\right)$ since $\cM (-,\cC):\sSet \to \sSet$ sends colimits to limits. To prove that $\overline{\nu}$ is a Kan fibration it suffices to show that
\begin{equation*}
 \nu: \cM\left(\cCartop S \times \cPyr {n,k}, \cC\right) \to \lim_{\Simplex m \to B(n,k)} \cM \left( \cCartop S \times \cPyr m , \cC\right)
\end{equation*}
is a Kan fibration as the target of $\overline{\nu}$ is a union of connected components of the target of $\nu$.

By left Kan extension along the Yoneda embedding one can extend $\cPyr{\bullet}$ to a functor $\cPyr {\bullet}: \sSet \to \sSet$. Let $\epsilon: \Delta \to \Delta$ to be the edgewise subdivison functor sending $[n]$ to $[n] \star [n]^{\rm op}$ (\cite{BarwickSpan} 2.5). Since $\cPyr n = \epsilon^*\Simplex n$ and both functors preserve colimits one has that $\cPyr\bullet = \epsilon^*$. From this we conclude that $\cPyr\bullet$ preserves products and monomorphisms.

It follows, therefore, that the target of $\nu$ is $\cM \left( \cCartop S \times \cPyr {B(n,k)}, \cC \right)$. The morphism $\nu$ is induced by the functor $\cPyr {B(n,k)} \to \cPyr {n,k}$, which is itself induced by the inclusion $\xymatrixcolsep{1pc}\xymatrix{B(n,k) \ar@{^{(}->}[r] & \Simplex n \times \Simplex k}$ and so is a monomorphism. Hence, by Lemma 3.1.3.6 of \cite{HTT}, $\nu$ is a Kan fibration.
\end{proof}

Next, denote by $V_{n,k}$ the sub-poset of $[n]\times[k]$ on those morphisms of the form $(\id,g)$. We say a functor $\cCartop S \times \cPyr {n,k} \to \cC$ is {\em vertically constant} if morphisms of the form $(\id, v)$ for $v \in \cPyr{V_{n,k}}$ are sent to equivalences.
\begin{defn}
\label{BispanqCat}
 Let $\cC$ be an $\infty$-category with finite limits and $(S_*,[n]) \in \FinDel$. Define $\Sp_{S,n}(\cC)$ to be the simplicial set having $k$-simplices the cartesian, vertically constant functors $\cCartop S \times \cPyr {n,k} \to \cC$.
\end{defn}

\begin{prop}
\label{BispanSS}
Let $\cC$ be an $\infty$-category with finite limits. Then for each $(S_*,[n])$ in $\FinDel$, the simplicial set $\Sp_{S,n}(\cC)$ is a quasi-category. Furthermore, the symmetric monoidal $\inftwo$-category of bispans $\csmtSpan \cC$ is given by the functor
  \begin{equation*}
   \FinDel \to \qCat^0, \quad (S_*, [n]) \mapsto \Sp_{S,n} (\cC),
   \end{equation*}
and so is semistrict.
\end{prop}
\begin{proof}
Recall that for a complete Segal space $\cX$ each vertex of $\cX_k$ determines a sequence of composable morphisms $[f_i]$ in the homotopy category of $\cX$. Then $\cX^\simeq_k$ is the full sub simplicial set of $\cX_k$ on those vertices having each $[f_i]$ invertible in the homotopy category (\cite{Lurieinftwo} 1.1.11). 

In particular, $\left(\overline{\csmtSpan \cC}(S_*,[0],\bullet)\right)^\simeq_k$ is the full sub simplicial set of $\cMap^{\rm cart}(\cCartop S \times \cPyr k,\cC)$ generated by those functors sending morphisms of the form $(\id, f)$ to equivalences. Since all of the objects in diagram \ref{LocPBCSS} are fibrant, Lemma \ref{FibProp} implies that the ordinary pullback in $\CSS$ is a complete Segal space presenting the homotopy pullback (\cite{HTT} A.2.4.4). We can therefore conclude that $\csmtSpan \cC(S_*,[n],[k])$ is the full sub simplicial set of $\cMap^{\rm cart}(\cCartop S \times \cPyr {n,k},\cC)$ generated by those functors sending morphisms of the form $(\id, v)$ for $v \in \cPyr{V_{n,k}}$ to equivalences.

The canonical equivalence $\cCSS \simeq \Catoo$ is presented by a Quillen equivalence (\cite{JoyalTierney} 4.11)
\begin{equation*}
 \radj[(p_1)*][(i_1)^*]{\CSS}{(\sSet)_{\rm Joyal}}, \quad (p_1)^*: X_{\bullet,\bullet} \mapsto X_{\bullet,0} \ .
\end{equation*}
By definition $\Sp_{S,n}(\cC) = (p_1)^* \csmtSpan \cC(S_*,[n],\bullet)$, which is a quasi-category since $(p_1)^*$ preserves fibrations.
\end{proof}

\begin{cor}
 \label{bispansop}
 For any $\infty$-category $\cC$ having finite limits its symmetric monoidal $\inftwo$-category of bispans is equivalent to its opposite.
\end{cor}
\begin{proof}
 By Proposition \ref{BispanSS}, the opposite of the symmetric monoidal $\inftwo$-category of bispans is
 \begin{equation*}
  \FinDel \to \qCat^0, \quad (S_*, [n]) \mapsto \Sp_{S,[n]^\op} (\cC), 
 \end{equation*}
where the quasi-category $\Sp_{S,[n]^\op}(\cC)$ has $k$-simplices the cartesian and vertically constant functors $\cCartop S \times \cPyr{[n]^\op,k} \to \cC$. Since $[n]$ is canonically isomorphic to $[n]^\op$, one has an equivalence $\cPyr{[n]^\op,k} \simeq \cPyr{n,k}$ and hence $\Sp_{S,[n]^\op}(\cC) \simeq \Sp_{S,n}(\cC)$.
\end{proof}

\subsection{Unstraightening \texorpdfstring{$\csmtSpan {N(C)}$}{Span(N(C))}}
\label{UnstrSpan}

The aim of this section is to give a description, sufficient for our purposes, of the unstraightening of the symmetric monoidal $\inftwo$-category of bispans for an ordinary category $C$ having finite limits. In general, determining the quasi-category $\Un(F)$ unstraightening a functor $F:\cD \to \Catoo$ can be quite difficult. This can, however, be done in the special case when $\cD=N(D)$ for an ordinary category $D$ and $F$ is semistrict (\cite{HTT} 3.2.5.2). In this case, the set of $k$-simplices $\Un(F)_k$ is the set of pairs
\begin{equation}
\label{RelNerve}
 \left( \sigma \in N_k(D), \ \left\{ \tau(J): J \to F(\sigma(\overline{J}))\right\}_{\emptyset \neq J \subset [k]} \right),
\end{equation}
where $\overline{J}$ is the maximal element of $J$. The family of functors $\tau$ must be such that for each $\emptyset \neq J \subset L \subset [k]$ the following diagram commutes
\begin{equation*}
  \xymatrixrowsep{1.1pc} \xymatrix{ J \ar[r]^-{\tau(J)} \ar@{^{(}->}[d] & F(\sigma(\overline{J})) \ar[d]^{F(\sigma(\overline{J} \leq \overline{L})} \\
  L \ar[r]_-{\tau(L)} & F(\sigma(\overline{L}))}
 \end{equation*}

 By Proposition \ref{BispanSS}, $\csmtSpan \cC$ is semistrict for any $\infty$-category $\cC$ having finite limits, in particular when $\cC=N(C)$ as we shall assume for the remainder of this section\footnote{In fact, we only ever make use of the particular case of $C = \sSetfop$. Performing a similar analysis as we present in this section for a general $\infty$-category would involve explicitly determining certain colimits in $\Catoo$. This is both considerably more difficult and unnecessary for our purposes.}. We can therefore apply the above to compute its unstraightening.

\begin{remark}
  As a slight abuse of notation we will not distinguish between an element $(f,\phi) \in N_k(\FinDel)$ and its opposite functor
  \begin{equation*}
   (f,\phi): [k]^\op \to \Finpt^\op \times \Delta.
  \end{equation*}
Furthermore, for $j \leq j' \in [k]$ we write the composite morphisms as
\begin{equation*}
 f_{j',j}: f(j) \to f(j')\in \Finpt \ \mathrm{and} \ \phi_{j',j}: \phi(j') \to \phi(j) \in \Delta.
\end{equation*}
 \end{remark}
 
Let $\Delta_{\mathrm{inj}}$ be the wide subcategory of $\Delta$ on the injective maps and let $I(k)$ denote the full subcategory of $[k]^{\rm op} \times \Cartne {[k]}$ on those objects $(j,J)$ such that $\emptyset \neq J \subset [j]$. For each $(f,\phi)\in N_k(\FinDel)$, denote by $E(f,\phi)$ the composite 
 \begin{equation*}
  \xymatrixcolsep{1.4pc}\xymatrix{I(k) \ar@{^{(}->}[r] & [k]^{\rm op} \times \Cartne {[k]} \ar[rr]^-{(f,\phi) \times O} & & \Finpt^{\rm op} \times \Delta^2 \ar[rrr]^-{\Cartop - \times \Pyr {\bullet,\bullet}} & & & \Cat,}
 \end{equation*}
 where $O$ is the forgetful map from $\Cartne {[k]}$ to $\Delta$.

 \begin{lem}
  The $k$-simplices of the unstraightening of $\csmtSpan{N(C)}$ are the pairs
\begin{equation*}
\left((f,\phi)\in N_k(\FinDel), \ \tau: E(f,\phi) \Rightarrow {\rm const}_C\right),
\end{equation*}
such that each component $\tau_{j,J}:\Cartop {f(j)} \times \Pyr{\phi(j),J} \to C$ is cartesian and vertically constant.
 \end{lem}
\begin{proof}
 By Eq.~\ref{RelNerve} and Proposition \ref{BispanSS}, a $k$-simplex in the unstraightening of $\csmtSpan{N(C)}$ is an element $(f,\phi) \in N_k(\FinDel)$ along with, for each $J \in \Cartne{[k]}$, a cartesian and vertically constant functor
 \begin{equation*}
  \tau(J): \Cartop{f(\overline{J})} \times \Pyr{\phi(\overline{J}), J} \to C
 \end{equation*}
such that for $ J \to L \in \Cartne{[k]}$ the following diagram commutes
\begin{equation}
\label{NatNerve}
 \xymatrixrowsep{1.1pc}  \xymatrix{ \Cartop{f(\overline{L})} \times \Pyr{\phi(\overline{L}), J} \ar[r] \ar[d] & \Cartop{f(\overline{J})} \times \Pyr{\phi(\overline{J}), J} \ar[d]^-{\tau(J)} \\
 \Cartop{f(\overline{L})} \times \Pyr{\phi(\overline{L}), L} \ar[r]_-{\tau(L)} & C}
\end{equation}
where the functor along the top is induced by $\overline{J} \leq \overline{L} \in [k]$ and the functor along the left is induced by $J \subset L$. Note that for every string of $n$ composable morphisms in $\Cartne{[k]}$ one can build an analogous commutative $(n+1)$-cube based on the commutativity of diagram \ref{NatNerve} and the functoriality of $\Cartop{-}\times \Pyr{\bullet,\bullet}$. 

From this family $\{ \tau(J)\}_{J \in \Cartne{[k]}}$ we shall construct a natural transformation $\tau:E(f,\phi) \Rightarrow {\rm const}_C$. For each $(j,J) \in I(k)$ one has a morphism $J \to [j] \in \Cartne{[k]}$, and hence the diagram \ref{NatNerve} with $L$ replaced by $[j]$ commutes. We define $\tau_{j,J}$ to be the composite functor along the diagonal of this diagram, which is cartesian and vertically constant by construction. Given a morphism $(j,J) \to (l,L) \in I(k)$ there is a triple of composable morphisms 
\begin{equation*}
J \to L \to [l] \to [j] \in \Cartne{[k]}.
\end{equation*}
The naturality of $\tau$ is a consequence of the commutative $4$-cube constructed from this triple of composable morphisms.

Conversely, given a natural transformation $\tau:E(f,\phi) \Rightarrow {\rm const}_C$ having cartesian and vertically constant components, define $\tau(J) := \tau_{\overline{J}, J}$. The commutativity of diagram \ref{NatNerve} follows directly from the naturality of $\tau$. 
\end{proof}

By the universal property of colimits in $\Cat$, a natural transformation $ E(f,\phi) \Rightarrow {\rm const}_C$ is equivalently a functor $\colim E(f,\phi) \to C$. Observe that the diagonal inclusion of $I(k)$ into the full subcategory of $([k]^{\rm op})^2\times \Cartne{[k]}$ on objects $(i,j,J)$, where $ J \subset [j]$, is final. Therefore $\colim E(f,\phi)$ is isomorphic to the product of the colimits over $\Cartop{f}: [k]^{\rm op} \to \Cat$ and the diagram
 \begin{equation*}
  \xymatrixcolsep{1.1pc}\xymatrix{I(k) \ar@{^{(}->}[r] & [k]^{\rm op} \times \Cartne {[k]} \ar[rr]^-{\phi\times O} & &  \Delta^2 \ar[rr]^{\Pyr {\bullet,\bullet}} & & \Cat}.
 \end{equation*}
 It follows that 
 \begin{equation*}
 \colim E(f,\phi) \simeq \Cartop{f(0)} \times \left( \colim_{(j,J)\in I(k)} \Pyr{L(\phi)_{j,J}} \right),
 \end{equation*}
  where $L(\phi)$ is the composite
\begin{equation}
 \label{Mdiag}
 L(\phi):\xymatrixcolsep{1.1pc}\xymatrix{I(k) \ar@{^{(}->}[r] & [k]^{\rm op} \times \Cartne {[k]} \ar[rr]^-{\phi\times O} & &  \Delta^2 \ar[rr]^{\bullet \times \bullet} & & \Cat}.
\end{equation}
It follows from Corollary \ref{Twistedcolim} that functors $\colim E(f,\phi) \to C$ are equivalently functors $\Cartop{f(0)} \times \Pyr{\colim L(\phi)} \to C$. 

Our first task in this section is to compute $\colim L(\phi)$. We shall then determine conditions under which a functor $\tau: \Cartop{f(0)} \times \Pyr{\colim L(\phi)} \to C$ restricts to cartesian and vertically constant functors $\tau_{j,J}$.

\paragraph{Computing the colimit of the diagram $L(\phi)$.} Recall that the {\em Grothendieck construction} of a functor $G:D \to \Cat$ is the category having as objects
\begin{equation*}
 \left\{ (g, d) \ | \ d \in D, \ g \in G(d)\right\},
\end{equation*}
and morphisms
\begin{equation*}
 (\alpha,\delta):(g,d) \to (g',d'), \  \delta:d \to d' \ {\rm and} \ \alpha:G(\delta)(d) \to d'.
\end{equation*}
Let $M_\phi$, for $\phi\in N_k(\Delta^{\rm op})$, be the category which is the Grothendieck construction of the functor
\begin{equation*}
 \xymatrix{ [k]^{\rm op} \ar[r]^-\phi & \Delta \ar@{^{(}->}[r] & \Cat}.
\end{equation*}
Explicitly, $M_\phi$ is the poset having object set 
\begin{equation*}
 \{ (a,b) \ | \ b \in[k]^{\rm op}, a \in \phi(b)\},
\end{equation*}
and ordering defined by declaring $(a,b) \leq (a',b')$ if and only if $b \leq b' \in [k]^{\rm op}$ and $\phi_{b,b'}(a) \leq a'$. 
\begin{exam}
For $\phi \in N_k(\Delta^{\rm op})$ the constant map on $[n]$, the poset $M_\phi$ is  $[n]\times[k]^\op$.
\end{exam}
\begin{exam}
 For the unique active morphism $\phi = ([2] \actmorL [1]) \in N_1(\Delta^{\rm op})$, the poset $M_\phi$ is
\begin{equation*}
 \xymatrixrowsep{.8pc}\xymatrixcolsep{1.1pc} \xymatrix{(0,1) \ar[d] \ar[rr] & & (1,1) \ar[d] \\
 (0,0) \ar[r] & (1,0) \ar[r] & (2,0)}
\end{equation*}
\end{exam}

\begin{lem}
\label{Mlem}
 The poset $M_\phi$ is the colimit of the diagram $L(\phi)$.
\end{lem}
\begin{proof}
 As the category $M_\phi$ is obtained via the Grothendieck construction it is the colimit of the diagram
 \begin{equation*}
  \xymatrixcolsep{1.1pc}\xymatrix{ \Pyr{([k]^{\rm op})} \ar[r] & [k]^{\rm op} \times [k] \ar[rr]^-{\phi \times [k]^{\rm op}_{-/}}&  & \Cat}. 
 \end{equation*}
 Observe that, since $[k]^{\rm op}_{i/} = [i]^{\rm op} \simeq [i]$, one can obtain this diagram by precomposing the diagram in Eq.~\ref{Mdiag} with the functor
 \begin{equation*}
  F: \Pyr{([k]^{\rm op})} \to I(k), \quad [j;i] \mapsto (j,[i]). 
 \end{equation*}
The functor $F$ is final, as for any $(j,J) \in I(k)$ with $\overline{J} = {\rm max} \, J$, the category $F^{(j,J)/}$ is the full subcategory of $\Pyr{([k]^{\rm op})}$ on those objects
\begin{equation*}
 \{[a;b] \ | \ \overline{J} \leq b \leq a \leq j\},
\end{equation*}
which is non-empty and connected.
\end{proof}

\paragraph{The cartesian and vertical constancy conditions.} Having determined that $\colim L(\phi) = M_\phi$ we shall now define conditions under which a functor $\tau: \Cartop{f(0)}\times\Pyr{M_\phi} \to C$ induces cartesian and vertically constant functors $\tau_{j,J}: \Cartop{f(j)} \times \Pyr{\phi(j),J} \to C$ for each $(j,J) \in I(k)$.

\begin{defn}
 Define $V_\phi$ to be the sub-poset of $M_\phi$ on those morphisms $(a,b) \to (\phi_{b,b'}(a),b')$ and $\Wedge {M_\phi}$ to be the full sub-poset of $\Pyr {M_\phi}$ on those intervals $[(a,b);(\phi_{b,b'}(a'),b')]$ satisfying $|b'-b| \leq 1$ and $|a'-a| \leq 1$.
\end{defn}

\begin{exam}
 For $\phi$ the constant map on $[n]$, one has that $V_\phi\simeq V_{n,k}$ and $\Wedge{M_\phi} \simeq \Wedge {n,k}$.
 \end{exam}
 
 \begin{exam}
  \label{Vphiact}
  For the unique active morphism $\phi = ([2] \actmorL [1]) \in N_1(\Delta^{\rm op})$, the poset $V_\phi$ is
\begin{equation*}
 \xymatrixrowsep{.8pc}\xymatrixcolsep{1.1pc} \xymatrix{(0,1) \ar[d]  & & (1,1) \ar[d] \\
 (0,0)  & (1,0)  & (2,0)}
\end{equation*}
and the poset $\Wedge {M_\phi}$ is
\begin{equation*}
 \xymatrixcolsep{.75pc}\xymatrixrowsep{.8pc} \xymatrix{[(0,1);(0,1)] & & \ar[ll] [(0,1);(1,1)] \ar[rr] & & [(1,1);(1,1)] \\
 [(0,1);(0,0)] \ar[u] \ar[d] & & [(0,1);(2,0)] \ar[ll] \ar[u] \ar[rr] \ar[rd] \ar[ld] & & [(1,1);(2,0)] \ar[u] \ar[d] \\
 [(0,0);(0,0)] & \ar[l] [(0,0);(1,0)] \ar[r] & [(1,0);(1,0)] & \ar[l] [(1,0);(2,0)] \ar[r] & [(2,0);(2,0)] }
\end{equation*}
\end{exam}

\begin{defn}
 Let $C$ be a category with finite limits and $(f,\phi)\in N_k(\FinDel)$. Then we say a functor $\tau: \Cartop{f(0)}\times \Pyr {M_\phi} \to C$ is:
 \begin{enumerate}
  \item {\em cartesian} if it is the right Kan extension of its restriction to $\Singop{f(0)} \times \Wedge{M_\phi}$.
  \item {\em vertically constant} if morphisms of the form $(\id, v)$ for $v \in \Pyr{V_\phi}$ are sent to isomorphisms.
 \end{enumerate}
\end{defn}
\begin{remark}
 When $(f,\phi) \in N_k(\FinDel)$ is the constant map on $(S_*, [n])$ this definition reduces to the notions already introduced for functors $\Cartop{S}\times\Pyr{n,k} \to C$.
\end{remark}

Since the categories $M_\phi$ for various $\phi \in N_k(\Delta^{\rm op})$ are obtained by the Grothendieck construction they are compatible in the following sense. For a natural transformation $\eta: \phi' \Rightarrow \phi$ one has a functor 
\begin{equation*}
M(\eta): M_{\phi'} \to M_{\phi}, \ (a,b) \mapsto (\eta_b(a), b)
\end{equation*}
and for a morphism $\gamma: [n] \to [k]$ there is a functor 
\begin{equation*}
M(\gamma): M_{\phi \gamma} \to M_{\phi}, \ (a,b) \mapsto (a, \gamma(b)).
\end{equation*}

\begin{lem}
\label{UnstrLem}
 Let $C$ be an category having finite limits, $\phi \in N_k(\Delta^{\rm op})$ and $F:\Pyr{M_\phi} \to C$ a cartesian functor. Then for a natural transformation $\eta: \phi' \Rightarrow \phi$ and a morphism $\gamma:[n] \to [k]$, the composite functors
 \begin{equation*}
  \xymatrixcolsep{.7pc} \xymatrix{ \Pyr{M_{\phi'}} \ar[rr]^-{M(\eta)} & & \Pyr{M_\phi} \ar[rr]^-F & & C & {\rm and} & \Pyr{M_{ \phi\gamma }} \ar[rr]^-{M(\gamma)} & & \Pyr{M_\phi} \ar[rr]^-F & & C}
 \end{equation*}
are cartesian. 
\end{lem}
 The proof of this Lemma is somewhat technical and the details are unnecessary for the remainder of the text. We therefore separate its proof into Section \ref{ProofUnstr}.

\begin{prop}
\label{UnstrSimps}
Let $C$ be a category with finite limits, and let $\Unstr C_k$ be the set
\begin{equation*}
 \left\{ \left((f,\phi)\in N_k(\FinDel), \, \tau: \Cartop{f(0)}\times\Pyr{M_\phi} \to C\right) \ | \ \tau \ {\rm cartesian}, \ {\rm vertically} \  {\rm constant}\right\}.
\end{equation*}
Then the sets $\Unstr C_k$ assemble into a sub simplicial set of the unstraightening of $\csmtSpan {N(C)}$.
\end{prop}
\begin{remark}
 In general, $\Unstr C$ is a proper sub simplicial set of the unstraightening. This will nonetheless suffice for our purposes as we will make use of this explicit form to map {\em into} the unstraightening. 
\end{remark}
\begin{proof}
 To begin, we observe that by Lemma \ref{UnstrLem}, if $\tau$ is cartesian then so is $\gamma^*\tau$ for any $\gamma:[n] \to [k]$, where
 \begin{equation*}
  \gamma^* \tau: \xymatrixcolsep{1.5pc}\xymatrix{\Cartop{f\gamma(0)}\times\Pyr{M_{\phi\gamma}} \ar[rrr]^-{\Cartop{f_{\gamma(0),0}}\times\Pyr{M(\gamma)}}& & &  \Cartop{f(0)}\times\Pyr{M_\phi} \ar[r]^-\tau & C}.
 \end{equation*}
 Furthermore, since $M(\gamma)$ maps $V_{ \phi\gamma}$ to $V_\phi$, the functor $\gamma^*\tau$ is vertically constant whenever $\tau$ is. Therefore, the sets $\Unstr C_k$ assemble into a simplicial set.
 
 By Lemma \ref{Mlem} the set of $k$-simplices of the unstraightening of $\csmtSpan{N(C)}$ consists of pairs
 \begin{equation*}
  \left\{ \left((f,\phi)\in N_k(\FinDel), \, \tau: \Cartop{f(0)}\times\Pyr{M_\phi} \to C\right)\right\}
 \end{equation*}
such that the induced functors $\tau_{j,J}$ are cartesian and vertically constant. To show that $\Unstr C$ is a sub simplicial set it therefore suffices to show that if $\tau$ is cartesian and vertically constant than so are the functors $\tau_{j,J}$. For the remainder of this proof we shall fix a cartesian, vertically constant functor $\tau$. 

Observe that $\phi(j) \times J^{\rm op}$ is $M_{c_{\phi(j)}}$, where $c_{\phi(j)}: J^{\rm op} \to \Cat$ is the constant functor on $\phi(j)$, and that one has a diagram
\begin{equation*}
\xymatrix{ J^{\rm op} \ar@{^{(}->}[d]_-\rho \ar[rd]^-{c_{\phi(j)}}& \ar@{}[ld]^(.5){}="a"^(.85){}="b" \ar@{=>}_\eta "a";"b" \\
 [k]^{\rm op} \ar[r]_-\phi & \Cat \, ,}
\end{equation*}
where $\eta$ is the natural transformation having components $\eta_i=\phi_{j,i}: \phi(j) \to \phi(i)$. Letting $\psi_{j,J}$ denote the composite functor
\begin{equation*}
 \psi_{j,J}:\xymatrixcolsep{.7pc}\xymatrix{ \phi(j) \times J \ar[r]^-\sim & M_{c_{\phi(j)}} \ar[rr]^-{M(\eta)} & & M_{\phi \rho} \ar[rr]^-{M(\rho)} & & M_\phi}, 
\end{equation*}
one has that $\tau_{j,J}$ is
\begin{equation*}
 \xymatrixcolsep{2.5pc} \xymatrix {\Cartop{f(j)}\times\Pyr{\phi(j),J} \ar[rr]^-{\Cartop{f_{j,0}} \times \Pyr{\psi_{j,J}}} & &  \Cartop{f(0)}\times\Pyr{M_\phi} \ar[r]^-\tau & C}.
\end{equation*}
It follows from Lemma \ref{UnstrLem} that $\tau_{j,J}$ is cartesian. The vertically constancy of $\tau_{j,J}$ follows from noting that $M(\eta)$ sends $V_{c_{\phi(j)}}$ to $V_{\phi \rho}$ and $M(\rho)$ sends $V_{\phi \rho}$ to $V_\phi$.
\end{proof}

\subsubsection{Proof of Lemma \ref{UnstrLem}}
\label{ProofUnstr}

Throughout this subsection we shall fix $\phi, \phi' \in N_k(\Delta^{\rm op})$, a natural transformation $\eta: \phi' \Rightarrow \phi$, a morphism $\gamma:[n] \to [k]$ and a cartesian functor $F: \Pyr{M_\phi} \to C$. We shall also make use of some notational shorthands. Set $\alpha= \Wedge{M_{\phi'}}$, $A=\Pyr{M_{\phi'}}$, $\theta = \Wedge{M_{\phi\gamma}}$, $\Theta=\Pyr{M_{\phi \gamma}}$, $\omega = \Wedge{M_\phi}$, $\Omega = \Pyr{M_\phi}$, $\xi=\Pyr{M(\eta)}$ and $\zeta=\Pyr{M(\gamma)}$.

We first show that $F \circ \xi$ is cartesian. Since $F$ is cartesian, for each $x \in A$, the object $F\xi(x)$ is the limit of the diagram
\begin{equation*}
 \xymatrix{ \omega(\eta)^{\xi(x)/} \ar@{^{(}->}[r] & \Omega \ar[r]^-F & C},
\end{equation*}
where $\omega(\eta)$ is the full subcategory of $\Omega$ containing $\omega$ as well as those objects $p$ such that $p \geq \xi(q)$ for some $q \in \alpha$. On the other hand, the right Kan extension of the restriction of $F \xi$ to $\alpha$ evaluated at $x \in A$ is the limit of the diagram
\begin{equation*}
 \xymatrix{ \alpha^{x/} \ar[r]^-{\xi^{x/}} & \omega(\eta)^{\xi(x)/} \ar@{^{(}->}[r] & \Omega \ar[r]^-F & C}.
\end{equation*}
To show that $F \xi$ is cartesian it therefore suffices to show that for each $x \in A$, the functor $\xi^{x/}$ is initial (\cite{MacLane} IX.3). That is, we must show that for each $y \in \omega(\eta)^{\xi(x)/}$, the poset
\begin{equation*}
 \xi^{x/}_{/y} = \left\{ z \in \alpha \ | \ z \geq x, \ \xi(z) \leq y \right\}
\end{equation*}
is non-empty and connected. 

Letting $y=[(c,d);(c',d')]\in \omega(\eta)^{\xi(x)/}$ we define the following
\begin{equation*}
 \hat{c} = \max \{i \in \phi'(d) \ | \ \eta_d(i) \leq c\} \ {\rm and} \ \hat{c}' = \min \{i \in \phi'(d) \ | \ \eta_{d'}(\phi'_{d,d'}(i)) \geq c'\}. 
\end{equation*}
There are a now two cases to consider. First, if $\hat{c} \leq \hat{c}'$ then $\hat{y} = [(\hat{c},d); (\phi_{d,d'}(\hat{c}'),d')]$ is the maximal element of $\alpha$ satisfying $\xi(\hat{y})\leq y$. Therefore $\xi^{x/}_{/y} = \alpha^{x/}_{/\hat{y}}$ and so $\xi^{x/}_{/y}$ is non-empty and connected as it has terminal object $\hat{y}$. Second, if $\hat{c}' \leq \hat{c}$ then from the inequalities
\begin{equation*}
 \phi_{d,d'}(c) \leq c' \leq \phi_{d,d'}\left(\eta_d(\hat{c}')\right) \leq \phi_{d,d'}\left( \eta_d(\hat{c})\right) \leq \phi_{d,d'}(c)
\end{equation*}
one concludes that $c'=\phi_{d,d'}(c)$, $\eta_d(\hat{c}') = \eta_d(\hat{c}) = c$, and $[\hat{c}';\hat{c}]$ is the largest subinterval of $\phi'(d)$ sent to $[c;c]$ under $\eta_d$. The poset $\xi^{x/}_{/y}$ is non-empty as it contains $[(\hat{c},d);(\phi'_{d,d'}(\hat{c}),d']$. Letting $[p;q]$ denote the maximal subinterval of $[\hat{c}';\hat{c}]$ such that $x \leq [(p,d);(q,d)]$, one has that for any $z \in \xi^{x/}_{/y}$ there is at least one $r \in [p;q]$ such that $z \leq [(r,d);(r,d)]$. As the poset $\xi^{x/}_{/y}$ contains the connected sub-poset 
\begin{equation*}
 \{[(u,d);(v,d)] \ | \ [u;v] \subset [p;q], \, |u-v|\leq 1\}
\end{equation*}
and each element maps into this sub-poset it follows that $\xi^{x/}_{/y}$ is connected.

Next, we will show that $F \circ \zeta$ is cartesian. The argument is quite similar to the above, and we shall recycle certain notation. It suffices to show that for each $x \in \Theta$ and $y \in \omega(\gamma)^{\zeta(x)/}$, the poset
\begin{equation*}
 \zeta^{x/}_{/y} = \left\{ z \in \theta \ | \ z \geq x, \ \zeta(z) \leq y \right\},
\end{equation*}
is non-empty and connected, where $\omega(\gamma)$ is full subcategory of $\Omega$ containing $\omega$ and everything above the image of $\theta$ under $\zeta$. Set $y = [(c,d);(c',d')] \in \omega(\gamma)^{\xi(x)/}$ and
\begin{equation*}
 \hat{d} = \max \{i \in [n]^{\rm op} \ | \gamma(i) \leq d \in [k]^{\rm op}\} \ {\rm and} \ \hat{d}' = \min \{i \in [n]^{\rm op} \ | \ \gamma(i) \geq d' \in [k]^{\rm op}\}.
\end{equation*}
There are again two cases. The first is when $\hat{d} \leq \hat{d}'$, then $\hat{y} = [(c,\hat{d});(c',\hat{d}')]$ is the maximal element of $\theta$ satisfying $\zeta(\hat{y})\leq y$, proving as above that $\zeta^{x/}_{/y}$ is non-empty and connected. Next, if $\hat{d}' \leq \hat{d}$ then $d'=d$ and $[\hat{d}';\hat{d}]$ is the largest subinterval of $[n]^{\rm op}$ sent to $[d;d]$. The same analysis as above, mutatis mutundi, shows that $\zeta^{x/}_{/y}$ is non-empty and connected.

\section{Simplicial objects define lax algebras in \texorpdfstring{$\csmtSpan{\cC}$}{Span(C)}}
\label{Sec:multialg}

In this section we present the main results of this paper. We first, in Section \ref{Sec:algcomb}, explicitly construct a symmetric monoidal lax functor
\begin{equation*}
 \xymatrix{\csmAlg \ar@{~>}[r]^-\alpha & \csmtSpan{\csSetfop}}, 
\end{equation*}
endowing the standard $1$-simplex $\Simplex{1}$ with the structure of a lax algebra. Then, in Section \ref{Sec:intospaces}, we show how the object of $1$-simplices $\cX_1$ of a simplicial object $\cX \in \cC_\Simp$ in an $\infty$-category having finite limits inherits the same structure from the universal property of $\Simplex{\bullet}$. Finally, we show in Section \ref{Sec:SegAssoc} that $\cX$ satisfies the $2$-Segal condition if and only if it inherits an algebra structure from $\Simplex{\bullet}$.

\subsection{The lax algebra structure on \texorpdfstring{$\Simplex 1$}{the standard 1-simplex}}
\label{Sec:algcomb}

Our first step in the construction of the lax algebra coming from a general simplicial object $\cX_\bullet \in \cC_\Simp$ will be to carry out the construction for the initial simplicial object $\Simplex{\bullet} \in \csSetfop$. According to Definitions \ref{DefnsmLax} and \ref{DefnLaxAlg}, this amounts to constructing a morphism of fibrations
\begin{equation*}
 \xymatrixrowsep{.8pc}\xymatrixcolsep{.9pc}\xymatrix{\Un(\csmAlg)\ar[rr]^-{\alpha} \ar[rd] & & \Un\left(\csmtSpan{\csSetfop}\right)\ar[ld] \\
  & \cFinDel & }
\end{equation*}
preserving cocartesian lifts of morphisms of the form $(f,\phi) \in \cFinpt\times(\Simpin)^{\rm op}$. Furthermore, the image of the object $\underline{1}$ in the fibre over $(\underline{1}_\ast,[0])$ must be $\Simplex 1$.

Before diving into the detailed construction of $\alpha$, let us first give an informal description. On objects, the lax functor $\alpha$ is simply 
\begin{equation*}
\alpha:X\mapsto X\cdot \Simplex 1 = \coprod_{x \in X} \Simplex 1.
\end{equation*}
Recall that a morphism in $\csmAlg$ is a function $p:X \to Y$ along with a linear ordering of $p^{-1}(y)$ for each $y \in Y$. The lax functor $\alpha$ sends the morphism $p$ to the morphism $\alpha(p) \in \csmtSpan{\csSetfop}$ given by the diagram
\begin{equation*}
 \xymatrixrowsep{.8pc}\xymatrix{ & \displaystyle \coprod_{y \in Y} \Simplex{|p^{-1}(y)|} & \\
 X\cdot \Simplex 1 \ar[ru]^-{\sigma} & & \ar[lu]_-{\lambda} Y\cdot \Simplex 1 \, ,}
\end{equation*}
The morphism $\lambda$ sends the $1$-simplex associated to the element $y \in Y$ to the long edge of the standard simplex $\Simplex{|p^{-1}(y)|}$. Note that one can label the edges along the spine of $\Simplex{|p^{-1}(y)|}$ by the elements of $p^{-1}(y)$ using the linear ordering. The morphism $\sigma$ sends the $1$-simplex associated to $x \in X$ to the appropriate edge along the spine of $\Simplex{|p^{-1}(p(x))|}$. 

\begin{exam}
For the morphism $m:\underline{2} \to \underline{1}$, the morphism $\alpha(m)$ in $\csmtSpan{\csSetfop}$ is given by the diagram in Eq.~\ref{DeltaProd}.
\end{exam}
\begin{exam}
  Consider the morphism $(m \amalg \id):\underline{3} = \underline{2} \amalg \underline{1} \to \underline{1} \amalg \underline{1} = \underline{2}$. Then $\alpha(m \amalg \id)$ is given by the diagram
  \begin{equation*}
   \includegraphics[height=1.45cm]{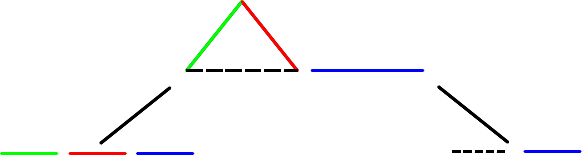} 
  \end{equation*}
\end{exam}

The lax structure on the functor $\alpha$ is given by associating to each pair of composable morphisms $\xymatrixcolsep{1.1pc}\xymatrix{X_0 \ar[r]^-{p_1} & X_1 \ar[r]^-{p_2} & X_2}$ in $\csmAlg$ a $2$-morphism in $\csmtSpan{\csSetfop}$ of the form
\begin{equation*}
 \xymatrixrowsep{.8pc}\xymatrix{ & \alpha(p_2p_1) \ar@{=}[d] & \\
 \alpha(X_0) \ar[ru] \ar[rd] \ar[r] & \alpha(p_2p_1) & \alpha(X_2) \ . \ar[lu] \ar[ld] \ar[l] & \\
  & \alpha(p_2) \coprod_{\alpha(X_1)} \alpha(p_1) \ar[u] & }
\end{equation*}

\begin{exam}
Consider the pair of composable morphisms $\xymatrixcolsep{1.7pc}\xymatrix{\underline{3} \ar[r]^-{m \amalg \id} & \underline{2} \ar[r]^-{m} & \underline{1}}$. The lax structure on $\alpha$ is given by the diagram
\begin{equation*}
 \includegraphics[height=4.5cm]{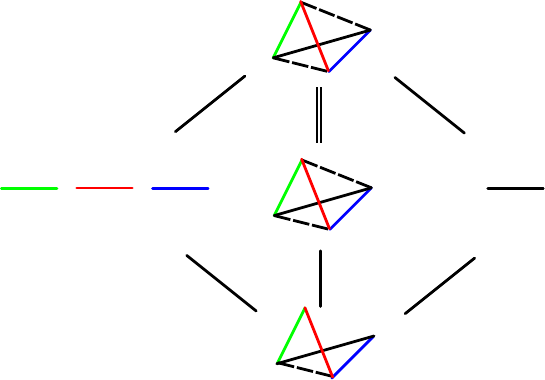} 
\end{equation*}
\end{exam}

\paragraph{Outline of the construction.} Recall from Section \ref{Sec:algprop} that $\csmAlg$ is semistrict by construction, and so it is straightforward to determine from Eq.~\ref{RelNerve} that its unstraightening is given by
\begin{equation*}
 \Un(\csmAlg)_k = \left\{ \left( (f,\phi)\in N_k(\FinDel), \ \theta:\Cart{f(0)}\times\phi(0) \to \Alg\right) \ | \ \theta \ {\rm cocartesian}\right\}.
\end{equation*}
From such data we will build a cartesian and vertically constant functor 
\begin{equation*}
\alpha(\theta):\Cartop{f(0)}\times \Pyr{M_\phi} \to \sSetfop.
\end{equation*}
By Proposition \ref{UnstrSimps} this defines a $k$-simplex in the unstraightening of $\csmtSpan{\csSetfop}$. To better understand the approach we take to the construction of $\alpha(\theta)$ it will be instructive to describe a special case.

Consider $\left((f,\phi),\theta\right) \in \Un(\csmAlg)_1$ where $f$ is constant on $\underline{1}_\ast$ and $\phi:[2] \actmorL [1]$ is the unique active morphism. Then the functor $\theta$ is a pair of composable morphisms
\begin{equation*}
\xymatrixcolsep{1.1pc}\xymatrix{X_0 \ar[r]^-{p_1} & X_1 \ar[r]^-{p_2} & X_2}. 
 \end{equation*}
 Since $\alpha(\theta)$ is cartesian it is determined by its restriction to $\Wedge{M_\phi}$, the diagram described in Example \ref{Vphiact}. The restriction of $\alpha(\theta)$ to $\Wedge{M_\phi}$ is
 \begin{equation*}
  \xymatrixrowsep{.8pc} \xymatrixcolsep{1.1pc} \xymatrix{\alpha(X_0) \ar[rr] \ar@{=}[d]  & & \alpha(p_2p_1) \ar@{=}[d] & & \alpha(X_2) \ar[ll] \ar@{=}[d] \\
  \alpha(X_0) \ar[rr] & & \alpha(p_2p_1) & & \alpha(X_2)\ar[ll] \\
  \alpha(X_0) \ar@{=}[u] \ar[r] & \alpha(p_1) \ar[ru] & \alpha(X_1) \ar[l] \ar[r] & \alpha(p_2) \ar[lu] & \alpha(X_2) \ar[l] \ar@{=}[u] }
 \end{equation*}
 Observe that all of the data in this diagram can be obtained from the following diagram
 \begin{equation}
 \label{alphabarexam}
  \xymatrixrowsep{.7pc}\xymatrixcolsep{1.2pc} \xymatrix{ & & \alpha(p_2 p_1) & & \\
  & \alpha(p_1) \ar[ru] & & \alpha(p_2) \ar[lu] & \\
  \alpha(X_0) \ar[ru] & & \alpha(X_1) \ar[lu] \ar[ru] & & \alpha(X_2) \ar[lu] }
 \end{equation}
which is a functor $\overline{\alpha}(\theta): \Pyr{\phi(0)} \to \sSetfop$. Specifically, the restriction of $\alpha(\theta)$ to $\Wedge{M_\phi}$ is obtained from $\overline{\alpha}(\theta)$ by restricting along a functor $\Wedge{M_\phi} \to \Pyr{\phi(0)}$.

In general, observe that for each $\phi \in N_k(\Delta^{\rm op})$ one has a natural transformation $\phi \Rightarrow c_{\phi(0)}$ having components $\phi_{i,0}$. The naturality of the Grothendieck construction implies that this natural transformation induces a functor 
\begin{equation*}
p_\phi:M_\phi \to \phi(0).
\end{equation*}
Our construction of $\alpha(\theta)$ for general $((f,\phi),\theta) \in \Un(\csmAlg)_k$ proceeds by first defining a functor
\begin{equation*}
 \overline{\alpha}(\theta):\Singop{f(0)}\times \Pyr{\phi(0)} \to \sSetfop
\end{equation*}
generalising the one in Eq.~\ref{alphabarexam}. The functor $\alpha(\theta)$ is then defined so that the following diagram is a right Kan extension
\begin{equation}
\label{RKEOutline}
 \xymatrixrowsep{.9pc}\xymatrixcolsep{2.2pc} \xymatrix{ \Cartop{f(0)}\times\Pyr{M_\phi} \ar[rrr]^-{\alpha(\theta)} & & & \sSetfop \\
 \Singop{f(0)}\times\Wedge{M_\phi} \ar@{^{(}->}[r] \ar@{^{(}->}[u] & \Singop{f(0)}\times\Pyr{M_\phi} \ar[r]_-{\id\times \Pyr{p_\phi}} & \Singop{f(0)}\times\Pyr{\phi(0)} \ar[r]_-{\overline{\alpha}(\theta)} & \sSetfop \ar@{=}[u]}
\end{equation}
The functor $\alpha(\theta)$ is manifestly cartesian, and is vertically constant as every morphism in $V_\phi$ is sent to the identity in $\phi(0)$ by $p_\phi$.

\paragraph{Construction of the functor $\overline{\alpha}(\theta)$.} By Theorem \ref{ErrOplax} it suffices to define a normal oplax functor
\begin{equation*}
 \overline{\alpha}(\theta): \Singop{f(0)}\times \phi(0) \nrightarrow \sp \sSetfop.
\end{equation*}

Recall from Eq.~\ref{nabladefn} that $\nabla$ is the category of spans of the form $\xymatrixcolsep{.7pc} \xymatrix{ \langle n \rangle & \ar[l] \langle k \rangle \ar@{>->}[r] & \langle m \rangle}$ in $\Delta_+$. The category of {\em levelwise finite nabla sets}, denoted $\Fin_\nabla$ is the category of functors $\nabla^\op \to \Fin$. From a morphism $p:X \to Y$ in $\Alg$ one can define the following diagram of levelwise finite nabla sets
\begin{equation*}
\xymatrixrowsep{.7pc} \xymatrix{ & \displaystyle \coprod_{y \in Y} \Simplexn{p^{-1}(y)} & \\
 X \ar[ru] & & \ar[lu] Y\, ,}
\end{equation*}
where $X$ and $Y$ are constant nabla sets and $\Simplexn{ n}$ is the nabla set represented by $\langle n \rangle$. The first and second map arise, respectively, from the following morphisms in $\nabla$:
\begin{equation*}
\xymatrixrowsep{.7pc} \xymatrixcolsep{.9pc} \xymatrix{ & \{ x \} \ar@{=}[ld] \ar@{>->}[rd] & & & p^{-1}(y) \ar[ld] \ar@{=}[rd] & \\
\{ x \} & & p^{-1}(p(x)) & \{ y \} & & p^{-1}(y)\, .}
\end{equation*}
 We can then apply the functor $\fG^*$ of Eq.~\ref{Augdef} to obtain a span in $\sSetfop$.

Now, define $\overline{\alpha}(\theta)$ on objects as
\begin{equation*}
 \overline{\alpha}(\theta) (s,i) = \fG^* \theta(s,i) = \theta(s,i) \cdot \Simplex 1.
\end{equation*}
To define $\overline{\alpha}(\theta)$ on a morphism $f:(s,i) \to (s,j)$ in $\Singop{f(0)}\times\phi(0)$, one applies the above construction to the morphism $\theta(f)$ in $\Alg$, yielding a span in $\sSetfop$
\begin{equation*}
 \xymatrixrowsep{.7pc} \xymatrix{ & \overline{\alpha}(\theta)(f) & \\
 \overline{\alpha}(\theta)(s,i) \ar[ru] & & \ar[lu] \overline{\alpha}(\theta)(s,j) \, . }
\end{equation*}

From a pair of composable morphisms $\xymatrixcolsep{1.2pc}\xymatrix{ X \ar[r]^-{p} & Y \ar[r]^-{q} & Z}$ in $\Alg$, one has a commutative diagram of nabla sets
\begin{equation*}
 \xymatrixrowsep{.8pc} \xymatrixcolsep{1.2pc} \xymatrix{ & & \displaystyle \coprod_{z \in Z} \Simplexn{(qp)^{-1}(z)} & & \\
 X \ar@/^1.1pc/[rru] \ar[r] & \displaystyle \coprod_{y \in Y} \Simplexn{p^{-1}(y)} \ar[ru]_-a  & Y \ar[r] \ar[l] &\ \displaystyle \coprod_{z \in Z} \Simplexn{ q^{-1}(z)}\ar[lu]^-b  & \ar[l] \ar@/_1.1pc/[llu] Z\, ,}
\end{equation*}
where $a$ and $b$ arise, respectively, from the following morphisms in $\nabla$:
\begin{equation*}
\xymatrixrowsep{.7pc} \xymatrixcolsep{1.7pc} \xymatrix{ & p^{-1}(y) \ar@{=}[ld] \ar@{>->}[rd] & & & (qp)^{-1}(z) \ar[ld]_-p \ar@{=}[rd] & \\
p^{-1}(y) & & (qp)^{-1}(q(y)) & q^{-1}(z) & & (qp)^{-1}(z) \, .}
\end{equation*}
Applying the functor $\fG^*$ of Eq.~\ref{Augdef} yields a diagram in $\sSetf$.

For a pair of composable morphisms $\xymatrixcolsep{1.2pc} \xymatrix{(s,i)\ar[r]^-f & (s,j) \ar[r]^-g & (s,k) }$ in $\Singop{f(0)}\times\phi(0)$ we define the corresponding component $\overline{A}(\theta)_{g,f}$ of the oplax structure on $\overline{\alpha}(\theta)$ as follows. Applying the construction of the preceding paragraph to the image under $\theta$ of the pair of composable morphisms one has, in particular, a commutative square
\begin{equation}
\label{OplaxComp}
 \xymatrixrowsep{.8pc} \xymatrix{ \overline{\alpha}(\theta)(s,j) \ar[r] \ar[d] & \overline{\alpha}(\theta)(g) \ar[d] \\
 \overline{\alpha}(\theta)(f) \ar[r] & \overline{\alpha}(\theta)(g \circ f) \, .}
\end{equation}
Then the component $\overline{A}(\theta)_{g,f}$ is the universal morphism
\begin{equation*}
 \overline{A}(\theta)_{g,f}: \overline{\alpha}(\theta)(g) \coprod_{\overline{\alpha}(\theta)(s,j)} \overline{\alpha}(\theta)(f) \to \overline{\alpha}(\theta)(g \circ f),
\end{equation*}
induced by the diagram in Eq.~\ref{OplaxComp}.

\begin{lem}
 For each $((f,\phi), \theta) \in \Un(\csmAlg)_k$, the above data defines a normal oplax functor and hence a functor
 \begin{equation*}
  \overline{\alpha}(\theta): \Singop{f(0)}\times \Pyr{\phi(0)} \to \sSetfop.
 \end{equation*}
\end{lem}
\begin{proof}
 The unitality conditions $\overline{A}(\theta)_{\id, f} = \id_{\overline{\alpha}(\theta)(f)}=\overline{A}(\theta)_{f, \id}$ are straightforward to verify. Given a triple $\xymatrixcolsep{1.1pc} \xymatrix{w \ar[r]^-g & x \ar[r]^-h & y \ar[r]^-i & z}$ of composable morphisms in $\Singop{f(0)}\times\phi(0)$, along similar lines as above one has a diagram in $\sSetf$
 \begin{equation*}
 \xymatrixrowsep{.7pc}\xymatrixcolsep{1.2pc} \xymatrix{ & & & \overline{\alpha}(\theta)(ihg) & & & \\
 & & \overline{\alpha}(\theta)(gh) \ar[ru] & & \overline{\alpha}(\theta)(ih) \ar[lu] & & \\
 & \overline{\alpha}(\theta)(g) \ar[ru] & & \overline{\alpha}(\theta)(h) \ar[ru] \ar[lu] & & \overline{\alpha}(\theta)(i) \ar[lu] & \\
 \overline{\alpha}(\theta)(w) \ar[ru] & & \overline{\alpha}(\theta)(x) \ar[ru] \ar[lu] & & \overline{\alpha}(\theta)(y) \ar[ru] \ar[lu] & & \overline{\alpha}(\theta)(z) \ar[lu] }
\end{equation*}
Then the associativity condition,
\begin{equation*}
  \overline{A}(\theta)_{i,hg} \circ \left(\id_{\overline{\alpha}(\theta)(i)} \coprod_{\overline{\alpha}(\theta)(y)} \overline{A}(\theta)_{h,g}\right)= \overline{A}(\theta)_{ih,g} \circ \left( \overline{A}(\theta)_{i,h} \coprod_{\overline{\alpha}(\theta)(x)} \id_{\overline{\alpha}(\theta)(g)}\right),
\end{equation*}
holds since both sides of the equation are the universal morphism for the colimit of the bottom two rows of the above diagram.
\end{proof}

\paragraph{The morphism $\alpha$ is a symmetric monoidal lax functor.} The diagram in Eq.~\ref{RKEOutline} defines, for each $(f,\phi) \in N_k(\FinDel)$ a functor
\begin{equation*}
 R_{(f,\phi)}: \Fun^{\rm cart}\left(\Singop{f(0)}\times \Pyr{\phi(0)}, \sSetfop\right) \to \Fun^{\rm cart}\left( \Cartop{f(0)}\times \Pyr{M_\phi}, \sSetfop \right).
\end{equation*}
We define $\alpha(\theta)$, for $((f,\phi),\theta)\in \Un(\csmAlg)_k$, to be
\begin{equation*}
 \alpha(\theta):= R_{(f,\phi)} \overline{\alpha}(\theta) \in \Unstr{\sSetfop}_k,
\end{equation*}
where $\Unstr{\sSetfop}$ is the sub simplicial set of the unstraightening of $\csmtSpan{\csSetfop}$ from Proposition \ref{UnstrSimps}.

\begin{lem}
 The assignment $((f,\phi),\theta) \mapsto \alpha(\theta)$ defines a morphism of simplicial sets 
 \begin{equation*}
  \alpha:\Un(\csmAlg) \to \Unstr{\sSetfop}.
 \end{equation*}
\end{lem}
\begin{proof}
 Fix $((f,\phi),\theta) \in \Un(\csmAlg)_k$ and a morphism $\gamma: [n] \to [k]$. We must show that $\gamma^* \alpha(\theta) = \alpha(\gamma^* \theta)$, that is, that the following diagram commutes, 

 \begin{equation*}
  \xymatrixrowsep{1.1pc}\xymatrixcolsep{2pc}\xymatrix{ \Cartop{f\gamma(0)}\times\Pyr{M_{\phi\gamma}} \ar[rd]^-{\alpha(\gamma^* \theta)} \ar[d] & \\
  \Cartop{f(0)} \times \Pyr{M_\phi} \ar[r]_-{\alpha(\theta)} & \sSetfop \, ,}
 \end{equation*}
where $\gamma^* \theta$ is the composite
\begin{equation*}
 \xymatrixcolsep{2pc}\xymatrix{\Cart{f\gamma(0)}\times \phi \gamma(0) \ar[r] & \Cart{f(0)}\times\phi(0) \ar[r]^-\theta & \Alg \, .}
\end{equation*}

Letting $\overline{\Singop{f(0)}}$ denote the full subcategory of $\Cartop{f(0)}$ containing $\Singop{f(0)}$ as well as those $U \subset f(0)$ such that $U \subset f_{\gamma(0),0}^{-1}(s)$ for some $s \in f\gamma(0)$, one has that the following diagram commutes
\begin{equation*}
 \xymatrixrowsep{.9pc} \xymatrix{\Fun^{\rm cart}\left(\Cartop{f(0)}\times\Pyr{M_\phi},\sSetfop\right) \ar[r]^{\gamma^*} & \Fun^{\rm cart}\left(\Cartop{f\gamma(0)}\times\Pyr{M_{\phi\gamma}},\sSetfop\right) \\
 \Fun^{\rm cart}\left( \overline{\Singop{f(0)}}\times \Pyr{\phi(0)}, \sSetfop\right) \ar[u]^-{R_{(f,\phi)}} \ar[r]_-{\gamma^*} &  \Fun^{\rm cart} \left(\Singop{f\gamma(0)}\times \Pyr{\phi\gamma(0)}, \sSetfop\right) \ar[u]_-{R_{(f\gamma, \phi\gamma)}} }
\end{equation*}
 It therefore suffices to show that for each $s \in f\gamma(0)$,
\begin{equation*}
 \overline{\alpha}(\gamma^*\theta)(s,-)= \coprod_{t \in f_{\gamma(0),0}^{-1}(s)} \gamma^*\overline{\alpha}(\theta)(t,-)
\end{equation*}
as normal oplax functors $\phi \gamma (0) \nrightarrow \sp{\sSetfop}$. This follows from the fact that $\theta$ is cocartesian.
\end{proof}

We have therefore constructed a morphism of fibrations
\begin{equation*}
 \xymatrixrowsep{.8pc}\xymatrixcolsep{.9pc}\xymatrix{\Un(\csmAlg)\ar[rr]^-{\alpha} \ar[rd] & & \Un\left(\csmtSpan{\csSetfop}\right)\ar[ld] \\
  & \cFinDel & }
\end{equation*}
Furthermore, the image of the object $\underline{1}$ in the fibre over $(\underline{1}_*, [0])$ is $\Simplex 1$. To show that the $\alpha$ endows $\Simplex 1$ with a lax algebra structure it remains only to show that $\alpha$ defines a symmetric monoidal lax functor.

\begin{prop}
 \label{PropCombAlg}
 The morphism of simplicial sets $\alpha$ is a symmetric monoidal lax functor 
 \begin{equation*}
 \alpha:\csmAlg \rightsquigarrow \csmtSpan{\csSetfop}.
 \end{equation*}
 Hence, $\alpha$ endows the standard $1$-simplex $\Simplex 1$ with the structure of a lax algebra..
\end{prop}
\begin{proof}
 We must show that for every $\left((f,\phi), \theta\right) \in \Un(\csmAlg)_1$ with $\phi$ inert the functor 
 \begin{equation*}
  \alpha(\theta)_{1,[1]}:\Cartop{f(1)}\times \Pyr{\phi(1),[1]}\to \sSetfop
 \end{equation*}
 defines an equivalence in the quasicategory $\Sp_{f(1),\phi(1)}(\csSetfop)$ (\cite{HTT} 3.2.5.2).

 For $\phi$ inert, the poset $M_\phi$ is of the form
 \begin{equation*}
  \xymatrixrowsep{.8pc}\xymatrixcolsep{2.5pc} \xymatrix{               
               &               & {\color{red} \bullet} \ar@[red][d] \ar@[red][r]  & {\color{red} \bullet} \ar@[red][r]\ar@[red][d] & \cdots \ar@[red][r] & {\color{red} \bullet} \ar@[red][r] \ar@[red][d] & {\color{red} \bullet} \ar@[red][d] &               &    \\
\bullet \ar[r] &  \cdots  \ar[r] & {\color{red} \bullet} \ar@[red][r]     &{\color{red} \bullet} \ar@[red][r]     & \cdots \ar@[red][r] & {\color{red} \bullet} \ar@[red][r] & {\color{red} \bullet} \ar[r] & \cdots \ar[r] & \bullet}
 \end{equation*}
 with the image of $\phi(1)\times[1]$ indicated in red.
 
Therefore for each $U  \in \Cartop{f(1)}$ the functor $\alpha(\theta)_{1,[1]}(U,-)$ is the right Kan extension of a diagram in $\sSetfop$ of the form
\begin{equation*}
 \xymatrixrowsep{.8pc}\xymatrixcolsep{1.3pc} \xymatrix{ 
 \bullet & \ar[l] \bullet \ar[r] & \bullet & \ar[l]\cdots\ar[r] & \bullet & \ar[l] \bullet \ar[r] & \bullet \\
 \bullet \ar@{=}[u] \ar@{=}[d] & \ar[l] \ar@{=}[u] \ar@{=}[d] \bullet \ar[r] & \ar@{=}[u] \ar@{=}[d] \bullet  & \ar[l] \ar[r]\cdots & \bullet \ar@{=}[u] \ar@{=}[d] & \ar[l] \ar@{=}[u] \ar@{=}[d] \bullet \ar[r] & \ar@{=}[u] \ar@{=}[d] \bullet \\
 \bullet & \ar[l] \bullet \ar[r] & \bullet & \ar[l] \ar[r]\cdots & \bullet & \ar[l] \bullet \ar[r] & \bullet } 
\end{equation*}
As the equivalences in $\Sp_{f(1),\phi(1)}$ are exactly those cartesian functors $\Cartop{f(1)}\times\Pyr{\phi(1),[1]} \to \sSetfop$ which send all morphisms of the form $(\id,\id,v)$ to equivalences (\cite{RuneSpans} 6.2) the claimed result follows. 
\end{proof}

\begin{remark}
\label{RemCombCoalg}
By Remark \ref{laxop}, the opposite of $\alpha$ gives a symmetric monoidal lax functor
\begin{equation*}
 \chi = \alpha^\op: \csmCoalg \rightsquigarrow \left(\csmtSpan\csSetfop\right)^\op.
\end{equation*}
However, Corollary \ref{bispansop} says that $\left(\csmtSpan\csSetfop\right)^\op \simeq \csmtSpan\csSetfop$. Therefore, as a dual to Proposition \ref{PropCombAlg}, we obtain a lax coalgebra structure on $\Simplex 1$.
\end{remark}

\subsection{The lax algebra structure inherited from \texorpdfstring{$\Simplex{1}$}{the standard 1-simplex}}
\label{Sec:intospaces}

Let $\cX \in \cC_\Simp$ be a simplicial object in an $\infty$-category having finite limits. We can now show how $\cX_1$ inherits a lax algebra structure in $\csmtSpan{\cC}$ from the one endowed upon standard $1$-simplex $\Simplex 1$ in Proposition \ref{PropCombAlg}. 

We make use of the following result due to Li-Bland.
\begin{thm}[\cite{LiBland} 4.1]
\label{LiBlandthm}
 The construction which assigns to an $\infty$-category having finite limits its symmetric monoidal $\inftwo$-category of bispans defines a functor
 \begin{equation*}
  \csmtSpan{-}: \Catoo^{\rm lex} \to \Bicatoo^\otimes.
 \end{equation*}
where $\Catoo^{\rm lex}$ is the $\infty$-category of $\infty$-categories having finite limits and finite limit preserving functors between them.
\end{thm}

Recall that any simplicial object $\cX_\bullet \in \cC_\Simp$ defines a finite limit preserving functor $\cX:\csSetfop \to \cC$  by right Kan extension,
\begin{equation*}
 \xymatrixrowsep{1.1pc} \xymatrix{ \csSetfop \ar[r]^-\cX & \cC \\
 \Simpop \ar@{^{(}->}[u] \ar[ru]_-\cX & }
\end{equation*}
By Theorem \ref{LiBlandthm} this induces a symmetric monoidal functor
\begin{equation*}
 \cX: \csmtSpan{\csSetfop} \to \csmtSpan{\cC}.
\end{equation*}
Then as an immediate corollary of Proposition \ref{PropCombAlg} one has the following.
\begin{thm}
\label{MainLaxAlg}
 Let $\cC$ be an $\infty$-category with finite limits and let $\cX_\bullet \in \cC_\Simp$ be a simplicial object. Then the composite
 \begin{equation*}
  \alpha_\cX: \xymatrix{\csmAlg \ar@{~>}[r]^-{\alpha} & \csmtSpan{\csSetfop} \ar[r]^-\cX & \csmtSpan{\cC}},
 \end{equation*}
is a symmetric monoidal lax functor endowing $\cX_1$ with the structure of a lax algebra.
\end{thm}

\begin{remark}
 Following Remark \ref{RemCombCoalg} we obtain a symmetric monoidal lax functor $\chi_\cX$ endowing $\cX_1$ with the structure of a lax coalgebra.
\end{remark}

\subsection{Associativity and the \texorpdfstring{$2$}{2}-Segal condition}
\label{Sec:SegAssoc}

Having equipped the object of $1$-simplices $\cX_1$ of a simplicial object $\cX_\bullet\in \cC_{\Simp}$ with a lax algebra structure in $\csmtSpan{\cC}$ in Theorem \ref{MainLaxAlg}, we now demonstrate that the $2$-Segal condition is exactly the right condition that enforces the associativity of this structure.

\begin{defn}{\cite{DK12} 2.3.2}
\label{2SegDefn}
A simplicial object $\cX \in \cC_\Simp$ is a $2$-Segal object if and only if for every $n\geq 3$ and every $0 \leq i < j \leq n$, the image of the squares 
\begin{equation}
\label{SegPO}
 \xymatrixrowsep{.9pc}\xymatrixcolsep{1.5pc} \xymatrix{ \Simplex{\{0,i\}} \ar[r] \ar[d] & \Simplex{\{0,i,\ldots,n\}} \ar[d] &{\rm and} & \Simplex{\{j,n\}} \ar[r] \ar[d] & \Simplex{\{0,\ldots,j\}} \ar[d] \\
 \Simplex{\{0,\ldots, i\}} \ar[r] & \Simplex n &  & \Simplex{\{j,\ldots,n\}} \ar[r] & \Simplex n}
\end{equation}
under $\cX$ are pullbacks in $\cC$ and for each $n \geq 2$ and $0 \leq i < n$, the image of the square
\begin{equation}
\label{SegUn}
 \xymatrixcolsep{.8pc}\xymatrixrowsep{.8pc}\xymatrix{\Simplex{\{i,i+1\}} \ar[r] \ar[d] & \Simplex n \ar[d] \\
 \Simplex{\{i\}} \ar[r] & \Simplex{n-1}}
\end{equation}
under $\cX$ is a pullback in $\cC$. 
In particular, a {\em $2$-Segal space} is a $2$-Segal object in $\cS$, the $\infty$-category of spaces.
\end{defn}
\begin{remark}
 What we call a $2$-Segal object is called a {\em unital $2$-Segal object} in \cite{DK12} and a {\em decomposition space} in \cite{KockI}.
\end{remark}
\begin{remark}
 For a simplicial object $\cX$ to be $2$-Segal it suffices for the images under $\cX$ of the squares in Eq.~\ref{SegPO} to be equivalences for $i=0$ or $j=n$ (\cite{DK12} 2.3.2).
\end{remark}

\begin{thm}
 \label{MainAlg}
 Let $\cC$ be an $\infty$-category with finite limits. Then a simplicial object $\cX \in \cC_\Simp$ is a $2$-Segal object if and only if the symmetric monoidal lax functor $\alpha_\cX$ of Theorem \ref{MainLaxAlg} endows $\cX_1$ with the structure of a algebra. 
 \end{thm}
\begin{proof}
We must show that $\cX$ is a $2$-Segal object if and only if $\alpha_\cX(\theta)_{1,[1]}$ is an equivalence in $\Sp_{f(1),\phi(1)}(\cC)$ for every $\left((f,\phi), \theta\right) \in \Un(\csmAlg)_1$. 
 
 Every active morphism $\psi:[n] \actmorR [m]$ can be decomposed as 
 \begin{equation*}
 \bigvee_{i=1}^{n} \psi_i:\xymatrixcolsep{1.4pc}\xymatrix{\displaystyle\bigvee_{i=1}^{n}[1] \ar@{->|}[r]& \displaystyle\bigvee_{i=1}^{n}[m_i]},
 \end{equation*}
 where $\psi_i$ is the unique active morphism $[1] \actmorR [m_i]$ and $[m] = \vee_i [m_i]$. Every inert morphism is of the form $[n] \inmorR [a] \vee [n] \vee [b]$. Since morphisms in $\Delta$ can be uniquely factored as an active morphism followed by an inert morphism, one can decompose the morphism $\phi:[m]\leftarrow[n] \in N_1(\Delta^{\rm op})$ as
 \begin{equation*}
  \phi= \bigvee_{i=1}^{n} \phi_i : [a_\phi] \vee \left(\bigvee_{i=1}^n [m_i]\right) [b_\phi] \leftarrow \bigvee_{i=1}^n [1].
 \end{equation*}

 One can therefore write the poset $M_\phi$ schematically as
 \begin{equation*}
  \includegraphics[width=9cm]{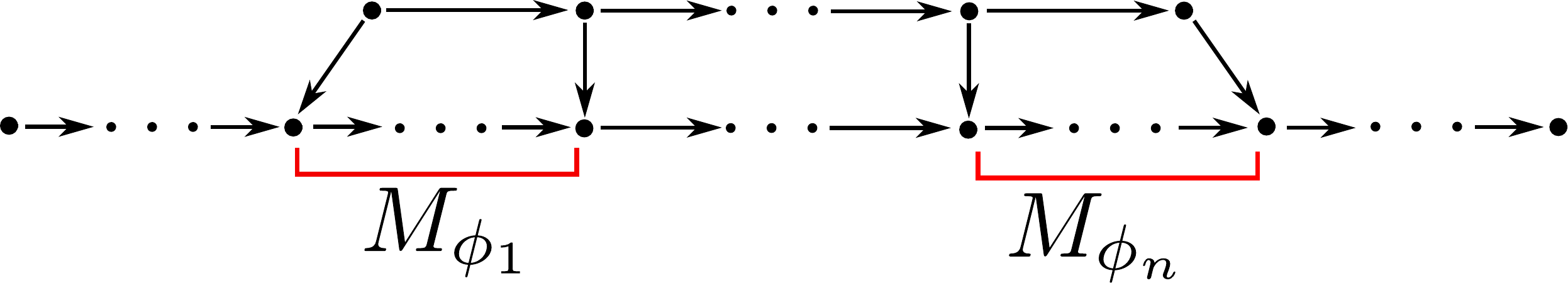}
 \end{equation*}
Denoting by $\theta_i: \Cart{f(0)} \times [1] \to \Alg$ the restriction of $\theta$ to the $i$-th summand of $\phi(1)$, it follows that for each $U \in \Cartop{f(1)}$, the functor $\alpha_\cX(\theta)_{\phi(1),[1]}(U,-)$ is the right Kan extension of a diagram in $\cC$ of the form
\begin{equation*}
  \includegraphics[height=2.25cm]{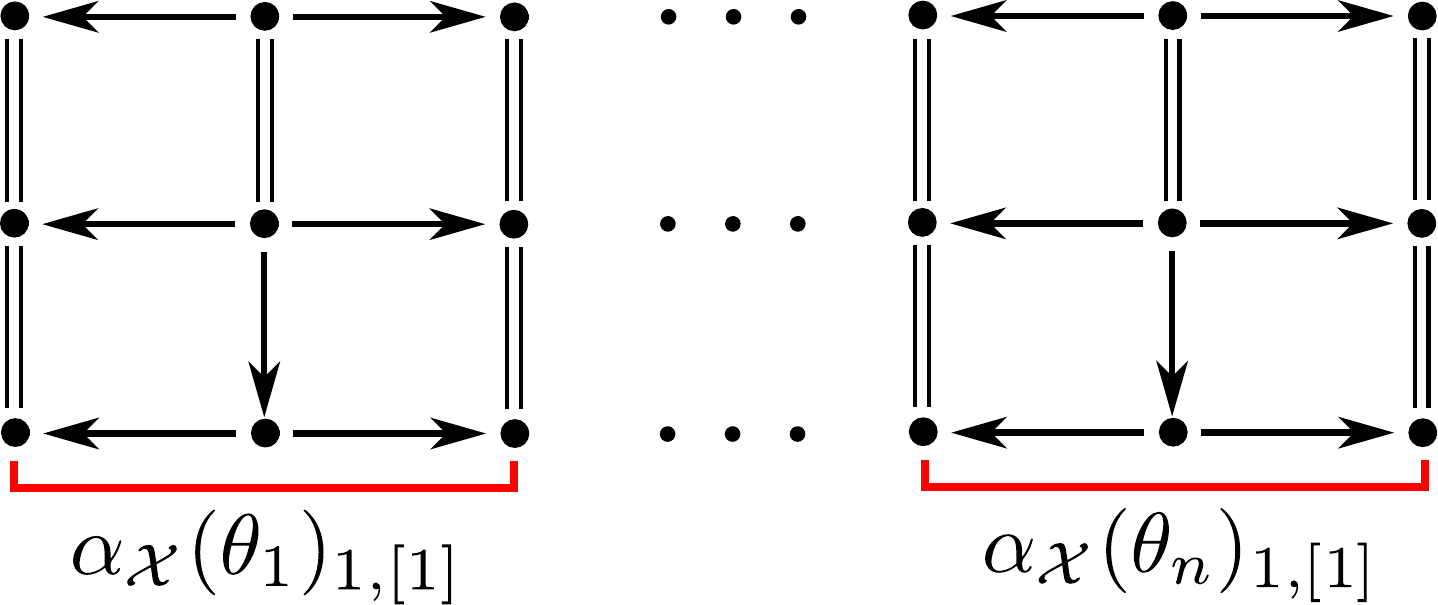}
 \end{equation*}
Therefore $\alpha_\cX(\theta)_{1,[1]}$ is an equivalence if and only if $\alpha_\cX(\theta_i)_{1,[1]}$ is an equivalence for each $i$. 

We conclude that $\alpha_\cX: \csmAlg \rightsquigarrow \csmtSpan{\cC}$ is a symmetric monoidal functor if and only if $\alpha_\cX(\theta)_{1,[1]}$ is an equivalence in $\Sp_{f(1),\phi(1)}(\cC)$ for every $\left((f,\phi), \theta\right) \in \Un(\csmAlg)_1$ where $\phi:[n] \actmorL [1]$ is the unique active morphism. By Lemmas \ref{sublemone} and \ref{sublemtwo}, proven below, this latter condition is equivalent to $\cX$ satisfying the $2$-Segal condition.
\end{proof}

To conclude our proof of Theorem \ref{MainAlg} we must prove two lemmas. 

\begin{lem}
 \label{sublemone}
 Let $\cX \in \cC_\Simp$ be a simplicial object in an $\infty$-category having finite limits. If, for every $\left((f,\phi), \theta\right) \in \Un(\csmAlg)_1$ with $\phi:[n]\actmorL [1]$ being the unique active morphism $\alpha_\cX(\theta)_{1,[1]}$ is an equivalence in $\Sp_{f(1),\phi(1)}(\cC)$, then $\cX$ satisfies the $2$-Segal condition.
\end{lem}
\begin{proof}
Consider, for each $n\geq 3$ and $i,j =1,\ldots, n-1$, the following pairs of composable morphisms in $\Alg$:
\begin{equation}
\label{SegPair}
 \xymatrixrowsep{.8pc} \xymatrix{\{1,\ldots,n\} \ar[r]^-{p^i_1} & \{i,\ldots,n\} \ar[r]^-{p^i_2} & \underline{1} \\
 \{1,\ldots,n\} \ar[r]_-{p^j_1} & \{1,\ldots,j+1\} \ar[r]_-{p^j_2} & \underline{1} \ .} 
\end{equation}
The morphism $p^i_1$ maps $a$ to $i$ when $1\leq a \leq i$ and $a$ otherwise, while $p^j_1$ maps $a$ to $j+1$ when $j+1 \leq a \leq n$ and $a$ otherwise. The linear orders on the fibres are induced by the ordering on $\{1,\ldots, n\}$. Observe that 
\begin{equation*}
\alpha(p^i_2) \coprod_{\alpha(\{i,\ldots,n\})} \alpha(p^i_1) \ \  {\rm and} \ \ \alpha(p^j_2) \coprod_{\alpha(\{1,\ldots,j+1\})} \alpha(p^j_1)
\end{equation*}
are, respectively, the pushouts of the left and right square of Eq.~\ref{SegPO}. Furthermore, $\alpha(p^i_2 p^i_1) = \Simplex n = \alpha(p^j_2 p^j_1)$, and the corresponding components of the lax structure on $\alpha$ agree with the ones arising from the squares in Eq.~\ref{SegPO}.

Next, consider for each $n \geq 2$ and $0\leq i < n$ the pair of composable morphisms in $\Alg$:
\begin{equation}
\label{UnitPair}
 \xymatrix{ \{1,\ldots, n-1\} \ar[r]^-{e^i_1} & \{1,\ldots, n\} \ar[r]^-{e^i_2} & \underline{1} \ ,}
\end{equation}
where $e^i_1$ is the evident map which skips the element $i+1$ in $\{1,\ldots, n\}$. Observe that 
\begin{equation*}
 \alpha(e^i_2) \coprod_{\alpha(\{1,\ldots,n\})} \alpha(e^i_1)
\end{equation*}
is the pushout of Eq.~\ref{SegUn}. We also have that $\alpha(e^i_2 e^i_1) = \Simplex{n-1}$, and the corresponding components of the lax structure on $\alpha$ agree with the ones arising from the square in Eq.~\ref{SegUn}.

Finally, consider $\left((f,\phi),\theta\right) \in \Un(\csmAlg)_1$, where $f$ is constant on $\underline{1}_\ast$ and $\phi:[2]\actmorL[1]$ is the unique active map. Then the functor $\theta$ is a pair of composable morphisms
\begin{equation*}
  \xymatrix{X_0 \ar[r]^-{p_1} & X_1 \ar[r]^-{p_2} & X_2}
\end{equation*}
and $\alpha_\cX(\theta)_{1,[1]}$ is the diagram
\begin{equation*}
 \xymatrixrowsep{.9pc} \xymatrix{ \alpha_\cX(X_1) & \alpha_\cX(p_2 p_1) \ar[r] \ar[l] & \alpha_\cX(X_2) \\
 \alpha_\cX(X_1) \ar@{=}[u] \ar@{=}[d] & \alpha_\cX(p_2 p_1) \ar[l] \ar@{=}[u] \ar[r] \ar[d]^-\delta & \alpha_\cX(X_2) \ar@{=}[u] \ar@{=}[d] \\
 \alpha_\cX(X_1) & \alpha_\cX(p_2) \times_{\alpha_\cX(X_1)} \alpha_\cX(p_1) \ar[r] \ar[l] & \alpha_\cX(X_2) }
\end{equation*}
The functor $\alpha_\cX(\theta)_{1,[1]}$ is an equivalence in $\Sp_{\underline{1},[1]}$ precisely when $\delta$ is an equivalence. 

Taking $\theta$ to be the pairs of morphisms defined in Eqs. \ref{SegPair} and \ref{UnitPair} it follows that $\cX$ is a $2$-Segal object.
\end{proof}

For the second of the two lemmas which complete Theorem \ref{MainAlg} we must make use of an equivalent formulation of the $2$-Segal condition due to G\'alvez-Carrillo--Kock--Tonks (\cite{KockI} 3). They show that $\cX \in \cC_\Simp$ is $2$-Segal if and only if the image under $\cX$ of every pushout square in $\Delta$ of the form
\begin{equation*}
 \xymatrixrowsep{.9pc} \xymatrix{ [n] \ar@{->|}[r] \ar@{>->}[d] & [m] \ar@{>->}[d] \\
 [k] \ar@{->|}[r] & [l]}
\end{equation*}
is a pullback in $\cC$.

\begin{lem}
 \label{sublemtwo}
 Let $\cX \in \cC_\Simp$ be a $2$-Segal object. Then $\alpha_\cX(\theta)_{1,[1]}$ is an equivalence in $\Sp_{f(1),\phi(1)}(\cC)$ for every $\left((f,\phi), \theta\right) \in \Un(\csmAlg)_1$ with $\phi:[n]\actmorL [1]$ being the unique active morphism.
\end{lem}
\begin{proof}
It suffices to show that for every sequence of composable morphisms in $\Alg$
\begin{equation*}
 \xymatrix{X_0 \ar[r]^-{p_1} & X_1 \ar[r]^-{p_2} & \cdots \ar[r]^{p_n} & X_n},
\end{equation*}
the image of
\begin{equation*}
 \alpha(p_n) \coprod_{\alpha(X_{n-1})} \cdots \coprod_{\alpha(X_1)} \alpha(p_1) \to \alpha(p_n\cdots p_1)
\end{equation*}
under $\cX$ is an equivalence. Since every morphism in $\Alg$ is a disjoint union of morphisms having target the singleton set, and the functors $\alpha$ and $\cX$ are symmetric monoidal, it suffices to consider the case when $X_n = \underline{1}$. The statement for general $n$ follows by iterating the special case of $n=2$.

First, in the trivial case of $X_1 = \underline{0}$, then also $X_0=\underline{0}$ and
\begin{equation*}
 \alpha(p_2)\coprod_{\alpha(X_1)} \alpha(p_1) = \Simplex 0 = \alpha(p_2p_1).
\end{equation*}
For $X_1 \neq \underline{0}$, write $X_1 = \{1,\ldots,k\}$ and denote by
\begin{equation*}
\kappa = \fG^{-1} (p_2)\  {\rm and} \ \kappa_i = \fG^{-1} \left(p_1|_{p_1^{-1}(i)}\right), \ i \in X_1,
\end{equation*}
where $\fG^{-1}(\langle n \rangle) = [n]$ is the inverse on objects of the functor in Eq.~\ref{Augdef}. Recall that $\Simplex \kappa$ can be thought of as a standard simplex having the edges along its spine labelled by the elements of $X_1$ according to the linear ordering defined by $p_2$. The pushout $\alpha(p_2) \coprod_{\alpha(X_1)} \alpha(p_1)$ is the simplicial set obtained by gluing each simplex $\Simplex{\kappa_i}$ by its long edge to the corresponding edge on the spine of $\Simplex \kappa$. Therefore $\alpha(p_2) \coprod_{\alpha(X_1)} \alpha(p_1)$ is the iterated pushout
\begin{equation*}
 \Simplex \kappa \coprod_{i=1}^k \Simplex{\kappa_i} =\left( \cdots \left(\left(\Simplex \kappa \coprod_{\Simplex{(0,1)}} \Simplex{\kappa_1}\right) \coprod_{\Simplex{(1,2)}} \Simplex{\kappa_2} \right) \cdots \right)\coprod_{\Simplex{(k-1,k)}} \Simplex{\kappa_k}\ .
\end{equation*}

Next, set $\kappa_{\leq i}$ to be the inductively defined pushouts in $\Delta$
\begin{equation}
\label{KappaPO}
 \xymatrixrowsep{.9pc} \xymatrix{[1] \ar@{->|}[r] \ar@{>->}[d]_-{g_0} & \kappa_1 \ar@{>->}[d]  & {\rm and} & [1] \ar@{->|}[r] \ar@{>->}[d]_-{g_i} & \kappa_{i+1} \ar@{>->}[d] \\
 \kappa \ar@{->|}[r] & \kappa_{\leq 1}  &  & \kappa_{\leq i} \ar@{->|}[r] & \kappa_{\leq i+1}}
\end{equation}
where $g_0$ has image the smallest two elements of $\kappa$, and $g_i$ has image the $k_i$'th and $k_i+1$'th elements of $\kappa_{\leq i}$ where $k_i = \sum_{j=1}^i |\kappa_j| -1$. Then since $\kappa_{\leq k} = \vee_{i=1}^k \kappa_i$ one has that $\Simplex{\kappa_{\leq k}} = \alpha(p_2p_1)$. Furthermore, the morphism $\alpha(p_2) \coprod_{\alpha(X_1)} \alpha(p_1) \to \alpha(p_2p_1)$ factors as the composite
\begin{equation}
\label{KappaComp}
 \xymatrixcolsep{2pc}\xymatrix{\Simplex \kappa \displaystyle \coprod_{i=1}^k \Simplex{\kappa_i}\ar[r] & \Simplex{\kappa_{\leq 1}} \displaystyle \coprod_{i=2}^k \Simplex{\kappa_i} \ar[r]  & \cdots \ar[r] & \Simplex{\kappa_{\leq k}} \,}
\end{equation}
where each morphism arises from the pushout squares in Eq.~\ref{KappaPO}.

It follows from the G\'alvez-Carrillo--Kock--Tonks form of the $2$-Segal condition that each morphism the sequence in Eq.~\ref{KappaComp} is sent to an equivalence under $\cX$, proving the claimed result.
\end{proof}

\bibliographystyle{plain}
\bibliography{refs}

 \end{document}